\documentclass[11pt,a4paper]{article}

\usepackage{xspace}
\usepackage{url}
\usepackage{mathtools}
\usepackage{amssymb}
\usepackage{amsthm}
\usepackage{empheq}
\usepackage{latexsym}
\usepackage[shortlabels]{enumitem}
\usepackage{eurosym}
\usepackage{dsfont}
\usepackage{appendix}
\usepackage{color} 
\usepackage[unicode]{hyperref}
\usepackage{frcursive}
\usepackage[utf8]{inputenc}
\usepackage[T1]{fontenc}
\usepackage{geometry}
\usepackage{multirow}
\usepackage{todonotes}
\usepackage{lmodern}
\usepackage{anyfontsize}
\usepackage{stmaryrd}
\usepackage{natbib}
\usepackage{cleveref}
\usepackage[english]{babel}
\usepackage[english=british]{csquotes}
\usepackage{mathtools}
\usepackage{accents}
\usepackage{color}
\usepackage{comment}
\usepackage{mdframed}
\usepackage{hyperref}
\usepackage{cases}
\usepackage{bm}
\usepackage[mathscr]{euscript}
\usepackage{mathrsfs}
\usepackage{float}

\mdfsetup{middlelinecolor=blue,middlelinewidth=2pt,linewidth=0pt,backgroundcolor=blue!15,,roundcorner=10pt}

\allowdisplaybreaks

\setcitestyle{numbers,open={[},close={]}}

\definecolor{red}{rgb}{0.7,0.15,0.15}
\definecolor{green}{rgb}{0,0.5,0}
\definecolor{blue}{rgb}{0,0,0.7}
\hypersetup{colorlinks, linkcolor={red}, citecolor={green}, urlcolor={blue}}
			
\makeatletter \@addtoreset{equation}{section}

\newtheorem{theorem}{Theorem}[section]

\newtheorem{corollary}[theorem]{Corollary}

\newtheorem{lemma}[theorem]{Lemma}
\newtheorem{proposition}[theorem]{Proposition}

\newtheorem{definition}[theorem]{Definition}
\theoremstyle{definition}
\newtheorem{remark}[theorem]{Remark}

\newcommand\cA{\mathcal A}

\newcommand\cF{\mathcal F}

\newcommand\cL{\mathcal L}

\newcommand\cP{\mathcal P}

\newcommand\cW{\mathcal W}
\newcommand\cX{\mathcal X}
\newcommand\cY{\mathcal Y}

\newcommand\sA{\mathscr A}

\newcommand\FP{\textnormal{FP}}
\newcommand\CFP{{\rm CFP}}
\newcommand\law{{\rm Law}}

\newcommand{\smalltext}[1]{\text{\fontsize{4}{4}\selectfont$#1$}}

\newcommand\1{\mathbf{1}}
\newcommand{\vertiii}[1]
{{\left\vert\kern-0.25ex\left\vert\kern-0.25ex\left\vert #1 \right\vert\kern-0.25ex\right\vert\kern-0.25ex\right\vert}}

\newcommand\fg{\mathfrak g}
\newcommand\cplbc{\mathfrak {\rm Cpl}_{\rm bc}}
\newcommand\cplmc{\mathfrak {\rm Cpl}_{\rm mc}}

\def \E{\mathbb{E}}
\def \F{\mathbb{F}}

\def \N{\mathbb{N}}

\def \P{\mathbb{P}}
\def \Q{\mathbb{Q}}
\def \R{\mathbb{R}}

\def \W{\mathbb{W}}
\def \X{\mathbb{X}}
\def \Y{\mathbb{Y}}

\def\Ac{\mathcal{A}}

\def\Fc{\mathcal{F}}

\def\Pc{\mathcal{P}}

\def\Wc{\mathcal{W}}
\def\Xc{\mathcal{X}}
\def\Yc{\mathcal{Y}}
\def\Zc{\mathcal{Z}}

\def\d{\mathrm{d}}

\begin{document}

\title{Absolutely Continuous Curves of Stochastic Processes}

\author{Beatrice Acciaio\footnote{ETH Zurich, Department of Mathematics, Switzerland, beatrice.acciaio@math.ethz.ch.} \quad Daniel Kr\v{s}ek\footnote{ETH Zurich, Department of Mathematics, Switzerland, daniel.krsek@math.ethz.ch.} \quad Gudmund Pammer\footnote{TU Graz,  Institute of Statistics, Austria, gudmund.pammer@tugraz.at.} \quad Marco Rodrigues\footnote{ETH Zurich, Department of Mathematics, Switzerland, marco.rodrigues@math.ethz.ch.}}
	
\date{\today}
	
\maketitle
		
\begin{abstract}
We study absolutely continuous curves in the adapted Wasserstein space of filtered processes. We provide a probabilistic representation of such curves as flows of adapted processes on a common filtered probability space, extending classical results to the adapted setting. Moreover, we characterize geodesics in this space and derive an adapted Benamou--Brenier-type formula by reformulating adapted optimal transport as an energy minimization problem. As an application, we obtain a Skorokhod-type representation for sequences of filtered processes under the adapted weak topology.


\vspace{5mm} 
\end{abstract}

\section{Introduction}

This paper extends results on absolutely continuous curves of probability measures with respect to the Wasserstein distance established by \citeauthor*{ambrosio2008gradient} \cite{ambrosio2008gradient} and \citeauthor*{lisini2007characterization} \cite{lisini2007characterization,lisini2016absolutely} to a dynamic framework involving filtered stochastic processes equipped with the adapted Wasserstein distance. 

Optimal transport offers a powerful mathematical framework that has been of increasing interest in recent years, and has revealed a geometric structure on the set of probability measures, which has proven to be valuable for applications in PDEs, image processing, machine learning, economics, and finance, among others. The study of absolutely continuous curves in the Wasserstein space serves as a cornerstone of the theory of gradient flows in this space and their wide applications, see for instance \cite{ambrosio2008gradient,jordan1998variational,lacker2023independent,mielke2004,otto2001,villani2009optimal}. One of the classical results asserts that, given a curve of measures on $\R^d$ that is absolutely continuous in the Wasserstein space, one can find a measure on absolutely continuous curves on $\R^d$ that ``couples'' the given flow of marginals. The converse is also true: given a measure on absolutely continuous curves, the corresponding projections necessarily form a continuous curve in the Wasserstein space. While the latter is straightforward, the former result is significantly nontrivial and allows to transition from absolute continuity of distributions to almost sure absolute continuity. That also leads to other crucial results, such as the probabilistic representation of absolutely continuous curves, the continuity equation, and the Benamou--Brenier formulation of the optimal transport problem.

While the geometry of the Wasserstein space and the related theory are well understood in $\R^d$ and, more generally, separable Hilbert spaces (see, e.g., \citeauthor*{ambrosio2008gradient} \cite{ambrosio2008gradient}), they pose challenges in more general metric spaces due to the lack of Hilbertian structure and potential infinite-dimensionality. A breakthrough work in this direction by \citeauthor*{lisini2007characterization} \cite{lisini2007characterization} studies absolutely continuous curves of measures on general metric spaces, results of whom were later generalized by \citeauthor*{stepanov2017three} \cite{stepanov2017three}. Motivated by these works, we study absolutely continuous curves of stochastic processes in the adapted Wasserstein space. We also refer to the forthcoming work of \citeauthor*{pinzi2025} \cite{pinzi2025}, in which they study absolutely continuous curves and the associated continuity equation in the space $\Pc(\Pc(\Xc))$, which is closely related to the adapted Wasserstein space.

As the Wasserstein distance and the weak topology have their limitations when dealing with laws of stochastic processes, illustrated, for example, by the discontinuity of the Doob decomposition or optimal stopping in this topology, we consider the adapted Wasserstein distance and the adapted weak topology; see e.g.\ \cite{BiTa19,NiSu20,pflug2014multistage,Ve70,Ve94}. This refinement of classical optimal transport theory, which takes into account the temporal structure of information by imposing certain non-anticipativity conditions on transport plans or couplings, has recently seen significant developments. In particular, it has been shown in the discrete-time setting that the adapted Wasserstein distance is the appropriate metric for comparing stochastic processes while accounting for their temporal structure. This is illustrated by their vast applications in stochastic analysis, mathematical finance, and machine learning; see for instance \cite{acciaio2024designing,AcBaZa20,acciaio2021weak,AcKrPa23a,BaBaBeEd19a,BaBaBeEd19b,BaWi23,BaBePa21,BaHa23,BoLiOb23,HoLeLiLySa23,JoPa23,PfPi12,sauldubois2024first,jiang2024sensitivity,jiang2025transfer}.

This framework allows us to define a geometry akin to the geometry of the Wasserstein space on the laws of stochastic processes or, more generally, on filtered processes -- the class of stochastic processes together with their filtrations. The present manuscript is devoted to the study of absolutely continuous curves of filtered processes, representing a first step toward a deeper understanding of the geometry of this complex space and laying the groundwork for a theory of gradient flows within it.

The first main result we present is a probabilistic representation of an absolutely continuous curve in the  metric space $(\FP_p,\Ac\Wc_p)$ of filtered processes with finite $p$-th moment, where $\Ac\Wc_p$ denotes the adapted $p$-Wasserstein distance, and by filtered process we understand a tuple
\[ \X=\Big(\Omega^\X,\Fc^\X,\F^\X=(\Fc^\X_t)_{t=1}^T,\P^\X,X=(X_t)_{t=1}^T \Big),\]
where $(\Omega^\X,\Fc^\X,\F^\X=(\Fc^\X_t)_{t=1}^T,\P^\X)$ is a filtered probability space, and $X$ is an $\F^\X$-adapted process with values in a given Polish path space $\Xc=\prod_{t=1}^T \Xc_t$.
Given a curve of filtered processes $[0,1] \ni u \mapsto \X^u \in \FP_p$ that is absolutely continuous in the adapted Wasserstein space, that is, $\X^\bullet\in AC^p(\FP_p)$, we show that there exists a common filtered probability space $(\Omega, \mathcal{F}, \mathbb{F}, \mathbb{P})$ and a flow of $\mathbb{F}$-adapted stochastic processes $Y^u$ defined on this space, such that at each point $u$ along the flow, the corresponding filtered stochastic process can be identified in the adapted Wasserstein space with the original one in the sense that $\Ac\Wc_p(\X^u,\Y^u)=0$. Moreover, the curve $u \mapsto Y^u(\omega)$ is an element of $AC_p(\Xc)$ for $\P$--almost every $\omega \in \Omega$. We also prove the converse: given such a probabilistic representation, the curve $u \mapsto \Y^u$ is necessarily absolutely continuous in the adapted Wasserstein space. In this way, we obtain a characterization of absolutely continuous curves in the adapted Wasserstein framework. This enables a transition of absolute continuity from the level of filtered processes to the level of individual particles on a fixed probability space, and vice versa.

Our approach is highly inspired by the work of \citeauthor{lisini2007characterization}~\cite{lisini2007characterization}, who showed that any absolutely continuous curve of probability measures \((\mu_u)_{u \in [0,1]}\) on a separable and complete metric space \(\mathcal{X}\) can be represented by a Borel probability measure \(\eta\) on the space \(AC([0,1]; \mathcal{X})\) of absolutely continuous curves, in the sense that the \(u\)-marginal of \(\eta\) is \(\mu_u\) for each \(u \in [0,1]\); see~\cite[Theorem 5]{lisini2007characterization}. \citeauthor{lisini2007characterization}'s proof relies on discretizing the curve \((\mu_u)_{u \in [0,1]}\), showing the tightness of the resulting interpolations, and extracting an accumulation point in the space of probability measures. Our method follows a similar line of reasoning, modified to the setting of the adapted Wasserstein distance.

We believe that the probabilistic representation we provide will prove useful in studying gradients in the adapted Wasserstein space. In particular, given a sufficiently regular functional on the adapted Wasserstein space, differentiating along an absolutely continuous curve may serve as a first step toward defining a suitable notion of derivative or gradient for such a functional, and our representation may prove useful in doing so. As this direction requires further investigation, we postpone it to future research.

Another central contribution of this work is the characterization of constant-speed geodesics in the space of filtered processes. In classical optimal transport, the Wasserstein space over a geodesic space inherits a geodesic structure: constant-speed minimizing geodesics represent the most efficient interpolations between probability measures, with mass transported along geodesics in the underlying space, a phenomenon known as displacement interpolation. This geometric perspective is fundamental in the study of evolution equations, particularly gradient flows of functionals such as the entropy, with broad applications in partial differential equations, machine learning, and statistics (see \cite{jordan1998variational,lacker2023independent,lambert2022variational,mokrov2021large,yan2024learning}).
We extend this geometric framework to the adapted setting, again drawing inspiration from the work of \citeauthor*{lisini2007characterization} \cite{lisini2007characterization}, by showing that geodesics in the adapted Wasserstein space can be realized as flows of random processes on a fixed filtered probability space, with the law of these flows concentrated on geodesics in the base space. Furthermore, we reformulate the adapted optimal transport problem as an energy minimization problem, leading to an adapted analogue of the Benamou--Brenier formula. As an application, we demonstrate how these results yield a Skorokhod-type representation for convergent sequences of filtered processes in the adapted weak topology.

The remainder of the paper is organized as follows. In \Cref{sec:prelim}, we introduce the notation and main concepts concerning absolutely continuous curves and filtered processes. In \Cref{sec_ac_curves}, we present the main results on the probabilistic representation of absolutely continuous curves of filtered processes. In \Cref{sec:Benamou-Brenier}, we apply these results to derive a Benamou--Brenier-type formulation of the adapted optimal transport problem. In \Cref{sec:skorokhod}, we apply our results to deduce a Skorokhod-type representation. Finally, \Cref{sec:proof_repres} contains the proof of the probabilistic representation. We collect auxiliary definitions and results in \Cref{sec:appendix}.

\section{Preliminaries and notation} \label{sec:prelim}

This section sets notation and recalls key definitions and results related to absolutely continuous curves, the adapted Wasserstein distance, and filtered processes.

In Euclidean space, absolutely continuous curves are differentiable almost everywhere, with their length given by the integral of the norm of their derivative. In metric spaces, despite the absence of a smooth structure, absolutely continuous curves retain a similar behavior: they admit a metric derivative almost everywhere, and their length is recovered by integrating this derivative; see \citeauthor*{ambrosio2008gradient}~\cite[Chapter 1]{ambrosio2008gradient} and \citeauthor*{burago2001course}~\cite[Section 2.7]{burago2001course}.

Let $(\Yc,d_{\Yc})$ be a separable and complete metric space. We denote by $\Pc(\Yc)$ the set of all probability measure on $\Yc$ endowed with the weak topology and, for $p \in [1,\infty)$, we write $\Pc_p(\Yc)$ for the subset of all measures with finite $p$-th moment. We further denote by $\Wc_p(\cdot,\cdot)$ the $p$-Wasserstein distance, so that 
\[ \Wc_p(\mu,\nu)= \inf_{\pi \in \textnormal{Cpl}(\mu,\nu)} \big(\E^\pi\big[d_{\Yc}^p(X,Y)\big]\big)^{1/p},\quad (\mu,\nu) \in \Pc_p(\Yc)^2,\] where $\textnormal{Cpl}(\mu,\nu)$ is the set of all couplings of $\mu$ and $\nu.$

We write $C(\Yc)$ for the set of all continuous functions $f: [0,1] \longrightarrow \Yc$, and we endow it with the supremum metric
\[ d_{\infty}(f,g)=\sup_{u \in [0,1]} d_{\Yc}(f(u),g(u)).\]
Further, for $p \in [0,\infty)$, we denote by $AC^p(\cY)$ the collection of paths $f : [0,1] \rightarrow \cY$ such that there exists $m \in L^p(\R)$, the space of $p$-integrable, real-valued functions on $[0,1]$, with
\begin{equation} \label{eqn:}
    d_{\cY}(f(u),f(v)) \leq \int_u^v m(r)\d r, \; 0 \leq u \leq v \leq 1,
\end{equation} and we call $AC^p(\cY)$ the space of absolutely continuous curves with finite $p$-energy. 

For a curve $f : [0,1] \rightarrow \cY$, we define the upper metric Dini derivative $|f^\prime| : [0,1] \rightarrow \cY$ by
\begin{equation*}
    |f^\prime|(u) = \limsup_{u \neq v \rightarrow u} \frac{d_\cY(f(u),f(v))}{|u-v|},
\end{equation*}
where we use the convention $0/0 = 0$. 
For all $f \in AC^p(\cY)$, we have
\begin{equation*}
    |f^\prime|(u) = \lim_{u \neq v \rightarrow u} \frac{d_{\cY}(f(u),f(v))}{|u-v|}, \; \textnormal{for a.e.~$u \in [0,1]$,}
\end{equation*} and refer to $|f^\prime|$ as the the metric derivative of $f$.
For $f \in AC^1(\cY)$, we can write
\begin{equation*}
    d_{\cY}(f(u),f(v)) \leq \int_u^v |f^\prime|(r)\d r, \; 0 \leq u \leq v \leq 1,
\end{equation*}
and thus a curve $f \in C(\cY)$ belongs to $AC^p(\cY)$ if and only if $|f^\prime| \in L^p(\R)$; see, e.g., \cite[Theorem 1.1.2]{ambrosio2008gradient}, where we also refer the interested reader for more details.

Finally, for $p \in [0,\infty),$ we define the Sobolev space
\[
    W^{1,p}(\cY) \coloneqq \bigg\{ [f]_\sim \in L^p(\cY) \,\bigg|\, \sup_{0 < h < 1} \int_0^{1-h} \bigg( \frac{d_\cY(f(u+h),f(u))}{h} \bigg)^p \d u < \infty \bigg\},
\]
where $L^p(\cY)$ denotes the collection of equivalence classes $[f]_\sim $ of $p$-integrable functions $f : [0,1] \longrightarrow \cY$ that agree Lebesgue--a.e. on $[0,1]$; see \Cref{sec_lp}. We recall the following result for later reference.

\begin{lemma}[Lemma 1 in \cite{lisini2007characterization}] \label{lem:sobolev_repres}
    Let $p \in (1,\infty)$, and let $\cY$ be a metric space.
    \begin{enumerate}[label=(\roman*)]
    \item If $f \in AC^p(\Yc),$ then $[f]_\sim \in W^{1,p}(\Yc).$
    \item For every $[f]_{\sim} \in W^{1,p}(\Yc)$, there exists $\tilde f \in AC^p(\Yc)$ satisfying $f(u) = \tilde f(u)$ for a.e.\ $u \in [0,1]$. In particular, $\tilde f$ is a continuous representative of the equivalence class $[f]_\sim$. Moreover, the map $W^{1,p}(\Yc) \ni [f]_\sim \longmapsto \tilde f \in C(\Yc)$ is Borel-measurable.
    \end{enumerate}
\end{lemma}

\subsection{Filtered processes} \label{sec:FP}

We fix separable and complete metric (thus Polish) spaces $(\cX_1,d_{\cX_1}), (\cX_2,d_{\cX_2}), \ldots,$\linebreak $(\cX_T,d_{\cX_T})$, and let $\cX$ be the metric space defined by $\cX = \cX_1 \times \cdots \times \cX_T$, endowed with $d_{\cX} = d_{\cX_1} \oplus \cdots \oplus d_{\cX_T}$, which will play the role of the path space. For $p \in [1,\infty)$,
we further define the $p$-metric on $\Xc$ by $d_{\cX,p} = (d^p_{\cX_1} \oplus \cdots \oplus d^p_{\cX_T})^{1/p}$.

\begin{definition}    A filtered process in $\cX$ is a tuple \[\X = \Big(\Omega^\X,\cF^\X,\P^\X,\F^\X = (\cF^\X_t)_{t = 1}^T,X = (X_t)_{t = 1}^T\Big)\] consisting of a filtered probability space $(\Omega^\X,\cF^\X,\P^\X,(\cF^\X_t)_{t = 1}^T)$ and an $\F^\X$-adapted stochastic process $X = (X_t)_{t = 1}^T$ with paths in $\cX$. The class of all filtered processes in $\cX$ is denoted by $\cF\cP(\cX)$ (or $\cF\cP$). For $p \in [1,\infty)$, we denote by $\cF\cP_p(\cX)$ (or $\cF\cP_p$) the class of all filtered processes $\X$ in $\cX$ satisfying $\E^{\P^\smalltext{\X}}[d_{\cX,p}(X,x_0)^p] < \infty$ for some (and then all) $x_0 \in \cX$.
\end{definition}

We always adjoin the trivial $\sigma$-field $\cF^\X_0 = \{\emptyset, \Omega^\X\}$ to $\F^\X$ whenever necessary.

\begin{remark}
    We note that the class of filtered processes does not constitute a set in the usual sense. If it were a set, one could form the product of all filtered processes in $\mathcal{FP}$, thereby constructing a new filtered process distinct from all elements of $\mathcal{FP}$, yet still belonging to $\mathcal{FP}$. However, this is not possible; see \textnormal{\Cref{sec::canonical_filtered_process}}.
\end{remark}

To simplify the notation, for filtered processes $\X$ and $\Y$, we denote
\[
    \cF^{\X,\Y} \coloneqq \cF^\X\otimes\cF^\Y \; \textnormal{and} \; \cF^{\X,\Y}_{s,t} \coloneqq \cF^\X_s \otimes \cF^\Y_t, \; s,t \in \{0,\ldots, T\}.
\]

\begin{definition}
    A coupling between $\X$ and $\Y$ is a probability measure $\pi$ on $(\Omega^\X \times \Omega^\Y, \cF^\X \otimes \cF^\Y)$ for which its marginal on $(\Omega^\X, \cF^\X)$ is $\P^\X$ and that on $(\Omega^\Y, \cF^\Y)$ is $\P^\Y$. The collection of couplings between $\X$ and $\Y$ is denoted by $\textnormal{Cpl}(\X,\Y)$. A coupling $\pi \in {\rm Cpl}(\X,\Y)$ is called:
\begin{enumerate}
\item[$(i)$] causal, if $\cF^{\X,\Y}_{T,0}$ and $\cF^{\X,\Y}_{0,t}$ are conditionally $\pi$-independent given $\cF^{\X,\Y}_{t,0}$ for every $t \in \{1,\ldots,T-1\}$;
\item[$(ii)$] anticausal, if $\cF^{\X,\Y}_{0,T}$ and $\cF^{\X,\Y}_{t,0}$ are conditionally $\pi$-independent given $\cF^{\X,\Y}_{0,t}$ for every $t \in \{1,\ldots,T-1\}$;
\item[$(iii)$] bicausal, if it is causal and anticausal.
\end{enumerate}
The collections of causal, anticausal and bicausal couplings between $\X$ and $\Y$ are denoted by $\textnormal{Cpl}_\textnormal{c}(\X,\Y)$, $\textnormal{Cpl}_\textnormal{ac}(\X,\Y)$ and $\textnormal{Cpl}_\textnormal{bc}(\X,\Y)$, respectively.

     Let further $\X^i=(\Omega^i,\Fc^i, \P^i,\F^i,X^i),$ $i \in \{1,\ldots,n\}$ be filtered processes. We say that $\pi \in \Pc(\prod_{i=1}^n \Omega^i)$ is a multicausal coupling of $\X^1,\ldots\X^n$ if its marginal on $(\Omega^i, \cF^i)$ is $\P^i$ for every $i \in \{1,\ldots,n\},$ and 
     \[\Fc_0^1 \otimes \cdots \otimes \Fc_0^{i-1} \otimes \Fc_T^{i} \otimes \Fc_0^{i+1} \otimes \cdots \otimes\Fc_0^n\quad {\rm and} \quad \Fc_t^1 \otimes \cdots \otimes \Fc_t^n\]
     are conditionally $\pi$-independent given $\Fc_0^1 \otimes \cdots \otimes \Fc_0^{i-1} \otimes \Fc_t^{i} \otimes \Fc_0^{i+1}\otimes \cdots \otimes \Fc_0^n$ for every $i \in \{1,\ldots,n\}$ and $t \in \{1,\ldots,T-1\}$.
\end{definition}

The adapted Wasserstein distance of order $p \in [1,\infty)$ between two filtered processes $\X$ and $\Y$ in $\cX$ is then defined as
\[
\cA\cW_p(\X,\Y) \coloneqq \inf_{\pi \in \textnormal{Cpl}_{\textnormal{bc}}(\X,\Y)} \big(\E^\pi\big[d_{\cX,p}(X,Y)^p\big]\big)^{1/p},\]
where \[ d_{\cX,p}(X,Y)(\omega^\X,\omega^\Y) = \bigg(\sum_{t = 1}^T d_{\cX_t}\big(X_t(\omega^\X),Y_t(\omega^\Y)\big)^p\bigg)^{1/p},\; (\omega^\X,\omega^\Y) \in \Omega^\X \times \Omega^\Y. \]
We recall from \cite[Theorem 3.1]{BaBePa21} that $\cA\cW_p$ behaves like a pseudo-metric \footnote{We say that $d$ is a pseudo-metric on $\cY$ if it satisfies all the conditions of being a metric except that $d_{\mathcal{Y}}(x, y) = 0$ might not imply $x = y.$ We point out that the terminology is inconsistent across related works, as some authors define a pseudo-metric as a function that satisfies all the properties of a metric except that they allow $d_\cY(x,y) = \infty$.} on the class $\cF\cP_p$. Defining the equivalence $\sim_{\FP}$ on $\Fc\Pc^2_p$ by \[ \X \sim_{\FP} \Y \iff  \inf_{\pi \in \textnormal{Cpl}_{\textnormal{bc}}(\X,\Y)} \E^\pi\big[d_{\cX,p}(X,Y) \wedge 1\big]=0 \] allows us to identify elements satisfying $\Ac\Wc_p(\X,\Y)=0$ and introduce the space of equivalence classes $\FP_p\coloneqq \Fc\Pc_p \big\vert_{\sim_{\FP}}$. Note that $(\FP_p,\Ac\Wc_p)$ is a metric space. From now on, we do not distinguish between filtered processes and their equivalence classes. This causes no ambiguity, as all properties studied in this paper are invariant under equivalence.

We recall the following Prokhorov-like compactness theorem, which will be useful in what follows.

\begin{lemma}[Theorem 5.1 in \cite{BaBePa21}]\label{lem::relatively_compact}
    Let $\cA\subseteq \FP_p$. Then  $\cA$ is precompact in $\FP_p$ if and only if $\{\textnormal{Law}_{\P^\smalltext{\X}}(X)\,\vert\, \X \in \cA\}$ is precompact in $\cP_p(\cX)$.
\end{lemma}

\section{Representation of curves in \texorpdfstring{$AC^p(\FP_p)$}.} \label{sec_ac_curves}

Since we will be working with absolutely continuous curves of filtered processes, we first briefly discuss these concepts in the specific setting of the adapted Wasserstein distance.

Let $(\X^u)_{u \in [0,1]}$ be a collection of elements in $\FP_p$. Then $(\X^u)_{u \in [0,1]}$ is an absolutely continuous curve in $\FP_p$ with finite $p$-th energy, denoted by $(\X^u)_{u \in [0,1]} \in AC^p(\FP_p),$ if and only if there exists $m \in L^p(\R)$ satisfying
\[ \Ac\Wc_p(\X^u,\X^v) \leq \int_u^v m(r) \mathrm d r,\quad 0 \leq u \leq v \leq 1.\]
For $p \in [1,\infty)$, the $p$-th upper metric Dini derivative $|\X^\prime|_p$ of a curve $[0,1] \ni u \mapsto \X^u \in \FP_p$ is defined as
\begin{equation*}
    |\X^\prime|_p(u) \coloneqq \limsup_{u \neq v \rightarrow u} \frac{\Ac\Wc_p(\X^u,\X^{v})}{|u-v|},\quad u \in [0,1].
\end{equation*}

Before stating the main theorems, we informally outline the main ideas behind them.
It is relatively straightforward to verify that, informally speaking, if we are given a filtered probability space $(\Omega, \Fc,\F, \P)$ supporting $\F$-adapted processes $\Y^u$, $u \in [0,1]$, such that the map $u \mapsto Y^u(\omega)$ is an absolutely continuous curve in $\Xc$ for $\P$--almost every $\omega \in \Omega$, then
\begin{equation} \label{eqn:probabil_repr_Yu}
u \longmapsto \Y^u \coloneqq \big( \Omega, \Fc, \F, \P, Y^u \big)
\end{equation}
defines an absolutely continuous curve in $\FP$.

In this section, we are primarily interested in the other implication: Given an absolutely continuous curve in $\FP$, is there a probabilistic representation of this curve akin to \eqref{eqn:probabil_repr_Yu}? Indeed, this is not a trivial question as absolutely continuous curves $(\X^u)_{u \in [0,1]}$ in $(\FP_p,\cA\cW_p)$ take values in equivalence classes of filtered processes. Thus, when choosing representatives of these equivalence classes, the curve may pass through different filtered probability spaces at different points. Consequently, one would have to deal with a continuum of potentially different base spaces. In this section, we show that it is indeed possible to construct a single filtered probability space $(\Omega, \cF, (\cF_t)_{t = 1}^T, \P)$ and a family $(Y^u)_{u \in [0,1]}$ of stochastic processes $(Y^u_t)_{t = 1}^T$ defined on it, such that at each point $u \in [0,1]$, $\Y^u \coloneqq (\Omega, \cF, (\cF_t)_{t = 1}^T, \P, (Y^u_t)_{t = 1}^T)$ satisfies $\Y^u \sim_{\FP}\X^u $ and $u \mapsto Y^u(\omega)$ is an absolutely continuous curve for $\P$--almost every $\omega.$

The proof we employ relies on discretization and piecewise linear approximation of the absolutely continuous curve, and a subsequent passage to the limit. For illustration, we briefly outline the key steps of the proof.
\begin{enumerate}[label=(\roman*)]
    \item Given a partition $0=u_0< u_1 <\cdots<u_{n-1}<u_n=1,$ we approximate the curve $u \mapsto \X^u$ by a piecewise linear curve as in \Cref{fig:discretization}, where we interpolate between two neighbouring elements $\X^{u_i}$ and $\X^{u_{i+1}}$ by the displacement interpolation in the adapted Wasserstein space. For each pair of neighbouring points $\X^{u_i}$ and $\X^{u_{i+1}}$, we thus select an optimal coupling $\pi^i \in \cplbc(\X^{u_i}, \X^{u_{i+1}})$.
    \item  These bicausal couplings $\pi^i$ can then be concatenated into a multicausal coupling $\gamma \in \cplmc(\X^0, \X^{u_1}, \ldots, \X^1)$ using Lemma A.5 from \cite{AcKrPa23a}.
    \item By means of $\gamma,$ we can find the sought representation for the vector of processes $(\X^0, \X^{u_1}, \ldots,\X^1)$ and thus also for the piecewise linear approximation of the curve.
    \item \label{itm:cluser_point} Then we choose a sequence of refining partitions and, using a suitable tightness criterion, show that we can pass to the limit, thereby obtaining a candidate representation of the curve.
    \item Finally, we verify that the resulting cluster point satisfies all the required properties.
\end{enumerate}

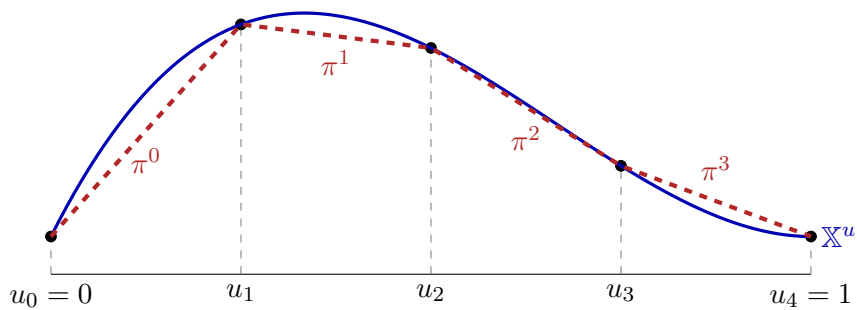
\begin{figure}[H]
    \centering
\begin{tikzpicture}[scale=10]
\draw[-] (0,0) -- (1,0);
\newcommand{\gammafun}[1]{0.05 + 2*(#1) - (2*(#1))^2 + 2*(#1)^3}
\foreach \x in {0, 0.25, 0.5, 0.75, 1.0} {
    \pgfmathsetmacro{\y}{0.05 + 2*(\x) - (2*(\x))^2 + 2*(\x)^3}
    \draw[gray, dashed] (\x,0) -- (\x,\y);
    \filldraw[black] (\x,\y) circle (0.2pt);
}
\draw[very thick, blue, domain=0:1, samples=100]
    plot (\x,{0.05 + 2*(\x) - (2*(\x))^2 + 2*(\x)^3})
    node[right] {$\X^u$};
\draw[ultra thick, red, dashed]
    plot coordinates {
        (0,{0.05 + 2*0 - (2*0)^2 + 2*0^3})
        (0.25,{0.05 + 2*0.25 - (2*0.25)^2 + 2*0.25^3})
        (0.5,{0.05 + 2*0.5 - (2*0.5)^2 + 2*0.5^3})
        (0.75,{0.05 + 2*0.75 - (2*0.75)^2 + 2*0.75^3})
        (1.0,{0.05 + 2*1.0 - (2*1.0)^2 + 2*1.0^3})
    };
\node[below] at (0,0) {$u_0=0$};
\node[below] at (1,0) {$u_4=1$};
\node[below] at (0.25,0) {$u_1$};
\node[below] at (0.5,0) {$u_2$};
\node[below] at (0.75,0) {$u_3$};

\node[below] at (0.125,0.18) {{\color{red} $\pi^0$}};
\node[below] at (0.375,0.31) {{\color{red} $\pi^1$}};
\node[below] at (0.625,0.21) {{\color{red} $\pi^2$}};
\node[below] at (0.875,0.17) {{\color{red} $\pi^3$}};
\end{tikzpicture}
\caption{Piecewise linear approximation of the curve $u \mapsto \mathbb{X}^u$.}
\label{fig:discretization}
\end{figure}

\begin{remark} \label{rem:accurves}
    Our results are similar in spirit to Theorems 4 and 5 in \cite{lisini2007characterization} (see also \cite[Theorem 8.2.1]{ambrosio2008gradient}) in that, given an absolutely continuous curve $(\X^u)_{u \in [0,1]},$ they enable us to couple the processes  $\X^u$ in a way that ensures that ``every particle moves absolutely continuously''. This allows us to transition from the absolute continuity of distributions (a macroscopic perspective) to the absolute continuity of individual particles (a microscopic perspective). However, we would like to emphasize that our approach differs in several key ways from a straightforward application of the results in \cite{lisini2007characterization}. Indeed, absolute continuity of the curve $(\X^u)_{u \in [0,1]}$ in $(\FP_p,\cA\cW_p)$ implies absolute continuity of the curve $u \mapsto \overline \Q^u\coloneqq{\rm Law}_{\P^u}({\rm ip}(\X^u))$ in $\Pc_p(\Zc)$ by \cite[Theorem 3.10]{BaBePa21}, see \Cref{sec::canonical_filtered_process} for the relevant definitions. It then follows, by application of \cite[Theorem 5]{lisini2007characterization}, that there exists $\eta \in \Pc(C(\Zc))$ concentrated on $AC^p(\Zc)$ such that $(f \mapsto f(u))_{\#}\eta=\overline \Q^u$. We thus obtain a ``coupling'' of $\X^u,$ $u \in [0,1]$ by means of $\eta.$ This construction is, however, different to the result we present here, as we have a flow of measures on the space $\Zc$ instead of working with a fixed filtered probability space with a single measure and a ``flow of random variables''. In the context of stochastic analysis, we believe that the latter is the natural representation and is more useful for applications. Moreover, we avoid using the canonical space $\Zc$ in our construction, which may be tedious and harder to work with.
\end{remark}

The following result shows that if a family of filtered processes evolves on the same filtered probability space, and if the law of this family of stochastic processes is concentrated on absolutely continuous trajectories in $\cX$, then the associated curve of filtered processes is absolutely continuous in the adapted Wasserstein distance.

\begin{theorem} \label{prop:AC1}
    Let $p \in [1,\infty),$ $(\Omega,\cF,(\cF_t)_{t=1}^T,\P)$ be a filtered probability space, and $Y^u = (Y^u_t)_{t=1}^T : \Omega \rightarrow \cX$, $u \in [0,1],$ a family of  $(\cF_t)_{t=1}^T$-adapted processes. Suppose that $\eta\coloneqq \law_{\P}((Y^u)_{u \in [0,1]}) \in \cP(C(\Xc))$ 
    is concentrated on $AC^p(\Xc)$ for some $p \in [1,\infty)$, that
\begin{align} 
    \Y^0 \coloneqq (\Omega,\cF,(\cF_t)_{t=1}^T,\P,(Y^0_t)_{t = 1}^T) &\in \FP_p, \quad \text{and}\nonumber \\ \int_{C([0,T],\Xc)}\int_0^1 |f^\prime|^p(s)\mathrm{d}s\,\eta(\mathrm{d}f)&<\infty. \label{eqn:finite_mom} 
\end{align} 
Then, the curve $u \mapsto \Y^u \coloneqq (\Omega,\cF,(\cF_t)_{t=1}^T,\P,(Y^u_t)_{t= 1}^T)$ is in $AC^p(\FP_p)$ and its metric derivative satisfies
\begin{equation}\label{eqn:lisini_derivative_FP_leq}
    |\Y^\prime|^p(u) 
    \leq \int_{C(\Xc)} |f^\prime |^p(u)\eta(\mathrm{d}f), \; \text{for {\rm a.e.}~$u \in [0,1]$.}
\end{equation}
\end{theorem}

\begin{proof} Since $\eta$ is concentrated on $AC^p(\Xc)$, the curve $u \mapsto \law_{\P}(Y^u)$ lies in \linebreak $AC^p(\Pc_p(\Xc))$ by a direct application of \cite[Theorem 4]{lisini2007characterization}. Consequently, by the same theorem, we have that 
\[
\eta \Big[\big\{f \in C(\Xc) \,\big|\, |f^\prime|(u) \text{ exists for a.e.\ }u \in [0,T]\big\} \Big]=1.
\] 

It is straightforward to check that  $({\rm id},{\rm id})_{\#}\P \in \cplbc(\Y^s,\Y^u)$  for
$0\leq s\leq u\leq 1$, and we thus obtain, by H\"older's inequality and Fubini's theorem, that
\begin{align*}
    \Ac\Wc_p^p(\Y^s,\Y^u) &\leq \int d^p_{\cX,p}(Y^s(\omega),Y^u(\omega)) \P(\mathrm{d}\omega)\\
    &\leq \int_{C(\Xc)} \Big( \int_s^u |f^\prime |(r)\mathrm{d}r \Big)^p \eta(\mathrm{d}f) \\
    &\leq |s-u|^{p-1} \int_{C(\Xc)}  \int_s^u |f^\prime |^p(r)\mathrm{d}r\, \eta(\mathrm{d}f)\\
    &=|s-u|^{p-1}  \int_s^u \int_{C(\Xc)} |f^\prime |^p(r)\eta(\mathrm{d}f)\mathrm{d}r \\
    &= \bigg(\int_s^u \d r\bigg)^{p-1}  \int_s^u \int_{C(\Xc)} |f^\prime |^p(r)\eta(\mathrm{d}f)\mathrm{d}r \\
    &\leq  \bigg(\int_s^u \int_{C(\Xc)} \max\{|f^\prime |^p(r), 1\}\eta(\mathrm{d}f)\mathrm{d}r\bigg)^p .
\end{align*} Because \eqref{eqn:finite_mom} holds, this shows absolute continuity of the curve $u \mapsto \Y^u$ in $\Fc\Pc_p.$ Moreover, the same chain of inequalities then also implies
\begin{align*}
    |\Y^\prime|^p(u) 
    &= \lim_{u \neq s \rightarrow u}\frac{\Ac\Wc_p^p(\Y^s,\Y^u)}{|s-u|^{p}} \leq \liminf_{u \neq s \rightarrow u}\frac{1}{|s-u|}  \int_s^u \int_{C(\Xc)} |f^\prime |^p(r)\eta(\mathrm{d}f)\mathrm{d}r\\ 
    & = \int_{C(\Xc)} |f^\prime |^p(u)\eta(\mathrm{d}f), \; \text{for {\rm a.e.}~$u \in [0,1]$},
\end{align*}
where the last equality follows from Lebesgue's differentiation theorem. This completes the proof.
\end{proof}

The main theorem of this section, which proves the converse implication, follows. We remark that its proof closely follows the proof of \cite[Theorem 5]{lisini2007characterization}, however, with significant modifications, as we aim to obtain a different representation than that described in \Cref{rem:accurves}.

\begin{theorem} \label{prop:AC_curve_in_FP}
    Let $p \in (1,\infty)$, and let $[0,1] \ni u \mapsto \X^u \in \FP_p$ be an absolutely continuous curve with finite $p$-energy.
    Then there exists a filtered probability space $(\Omega,\cF,(\cF_t)_{t=1}^T,\P)$ and a family of process $Y^u = (Y^u_t)_{t=1}^T : \Omega \rightarrow \cX$, $u \in [0,1]$, such that the filtered processes
    \begin{equation}
        \Y^u = (\Omega,\cF,(\cF_t)_{t=1}^T,\P, (Y^u_t)_{t=1}^T), \; u \in [0,1],
    \end{equation}
    satisfy $\cA\cW_p(\Y^u,\X^u) = 0$ for every $u \in [0,1]$, the paths $u \mapsto Y^u(\omega)$ are continuous for every $\omega \in \Omega$, and $\eta\coloneqq \law_{\P} ((Y^u)_{u \in [0,1]}) \in \cP(C(\cX))$  is concentrated on $AC^p(\Xc).$ Moreover, 
    \begin{equation}\label{eqn:lisini_derivative_FP}
        |\Y^\prime|^p(u) = \int_{C(\cX)} |f^\prime|^p(u)\eta(\d f), \; \text{for {\rm a.e.}~$u \in [0,1]$.}
    \end{equation}
\end{theorem}
\begin{proof} The proof is lengthy and is thus postponed to \Cref{sec:proof_repres}.
\end{proof}

\begin{remark} From the proof of Theorem~\ref{prop:AC_curve_in_FP}, it is clear that, when $\X^\bullet \in AC^p(\FP_p)$, we in particular obtain that 
\[ \Y\coloneqq \big( \Omega, \Fc, (\Fc_t)_{t=1}^T,\P, (Y_t^\bullet)_{t=1}^T \big)\] is a filtered process with $Y_t^\bullet \in  AC^p(\Xc_t)$, that is, $\Y \in \Fc\Pc\big(\prod_{t=1}^T AC^p(\Xc_t)\big).$
\end{remark}

\section{A Benamou--Brenier-type result} \label{sec:Benamou-Brenier}

This section focuses on geodesics in \((\FP_p, \cA\cW_p)\). In a metric space, a constant-speed minimizing geodesic is a curve that connects two points while tracing the shortest path between them at uniform speed. This generalizes the notion of straight lines from Euclidean geometry to more abstract spaces, where distances can be measured without relying on a linear or smooth structure.
It was previously noted in \cite[Theorem 5.10]{BaBePa21} that $(\FP_p, \cA\cW_p)$ is a geodesic space if all $\cX_t$ for $t \in \{1,\ldots,T\}$ are geodesic. The main result in this section is a complete characterization of geodesics in \((\FP_p, \cA\cW_p)\) and a proof that, just like in the classical Wasserstein setting, they follow displacement interpolation by moving along geodesics in $\cX$. We begin by recalling basic definitions in the general setting of metric spaces.

Let $(\Yc,d_{\Yc})$ be a metric space. The length $L(\gamma)$ of a continuous curve $\gamma : [0,1] \to \Yc$ is defined as
\[
    L(\gamma) \coloneqq \inf\bigg\{\sum_{k = 0}^{N-1} d_{\Yc}(\gamma(u_{k+1}),\gamma(u_k)) \,\bigg|\, N \in \mathbb{N}, \, 0 \leq u_0 < u_1 < \dots < u_N \leq 1\bigg\}.
\] 
If, for all $x,y \in \Yc$, the distance $d_{\Yc}$ satisfies
\[
    d_{\Yc}(x,y) = \inf L(\gamma),
\]
with the infimum taken over all continuous curves $\gamma$ connecting $x$ to $y$, we call $(\Yc,d_{\Yc})$ a length space.
A continuous curve $\gamma : [0,1] \to \Yc$ is called a minimizing geodesic if it satisfies
\[
    d_{\Yc}(\gamma(0), \gamma(1)) = L(\gamma).
\]
Moreover, if $\gamma$ additionally satisfies the condition
\[
    d_{\Yc}(\gamma(u),\gamma(v)) = |u-v| d_{\Yc}(\gamma(1),\gamma(0)), \quad u,  v \in [0,1],
\]
then we say that $\gamma$ is a constant-speed minimizing geodesic. Lastly, if for every pair of points $x,y \in \Yc$ there exists a $($resp. unique$)$ constant-speed geodesic connecting them, then $(\Yc,d_\cY)$ is called a (resp. uniquely) geodesic space.
We let \[G(\Xc):=\{ f :[0,1]\longrightarrow \Xc \,\vert\, f \text{ is a constant-speed minimizing geodesic of }\Xc \} \] and notice it is a closed subset of $C(\Xc)$, when the latter is endowed with the topology of uniform convergence. We then endow $G(\cX)$ the inherited metric $d_{G(\Xc)}(f,g)\coloneqq\sup_{u \in [0,1]}d_{\Xc}(f(u),g(u)),$ $(f,g) \in G(\Xc)^2.$

\begin{theorem} \label{lem:Benamou_Brenier_1} Let $p \in (0,\infty)$. A curve $[0,1] \ni u \mapsto \X^u \in \FP_p$ is a constant speed geodesic if and only if there exists a filtered space $(\Omega,\Fc,(\Fc_t)_{t=1}^T,\P)$ with $(\Fc_t)_{t=1}^T$-adapted processes $(Y^u)_{u \in [0,1]}$ such that
\begin{enumerate}[label=(\roman*)]
    \item \label{item:BB1}  $\X^u\sim \Y^u\coloneqq(\Omega,\Fc,(\Fc_t)_{t=1}^T,\P,Y^u),$ $u \in [0,1];$
    \item \label{item:BB2} $\eta \coloneqq \law_{\P}((Y^u)_{u \in [0,1]})$ is concentrated on $G(\Xc);$
    \item \label{item:BB3} $\displaystyle \Ac\Wc_p^p(\X^0,\X^1)=\int d^p_{\cX,p}(Y^0(\omega),Y^1(\omega)) \P(\mathrm{d}\omega).$
\end{enumerate}
\end{theorem}

\begin{remark}
\begin{enumerate}[label = (\roman*)]
    \item It is clear that \Cref{item:BB2} in \Cref{lem:Benamou_Brenier_1} means that the paths $u \mapsto Y^u(\omega)$ are absolutely continuous for $\P$--almost all $\omega \in \Omega.$
        
    \item Let $\Xc$ be a uniquely geodesic space and for $x,y \in \Xc$ denote $\gamma^{x,y} : [0,1] \longmapsto \Xc$ the constant-speed minimizing geodesic connecting $x$ and $y.$ In this particular case, we obtain that, for $t \in [0,T]$,
    \begin{align*} \int d^p_{\cX,p}\big(\gamma^{Y^0(\omega),Y^1(\omega)}(t),Y^t(\omega)\big) \P(\mathrm{d} \omega)= \int d^p_{\cX,p}\big(\gamma^{f(0),f(1)}(t),f(t) \big) \eta(\mathrm{d}f)=0,\end{align*} where the last equality follows from the fact that $f(t)=\gamma^{f(0),f(1)}(t)$ for any $f \in G(\Xc),$ which holds by uniqueness of $\gamma^{f(0),f(1)}$. Hence, it holds \begin{align}\label{eqn:const_spd_geo} Y^t = \gamma^{Y^0,Y^1}(t),\quad \P \text{--a.s.\ for every }t \in [0,T]. \end{align}
    As both sides in \eqref{eqn:const_spd_geo} are $\P$--a.s.\ continuous in $t,$ the statement holds, in fact, for every $t \in [0,T]$ $\P$--almost surely. 
    
 To be precise, $\gamma$ above has to be replaced by a measure $\gamma_{\P}$ as follows: the map $\gamma^{\cdot,\cdot}: \Xc \times \Xc \longrightarrow G(\Xc)$ admits a Borel-measurable version $\gamma_{\P}^{\cdot,\cdot} : \Xc\times \Xc \longrightarrow G(\Xc)$ such that \[\gamma_{\P}^{Y^0,Y^1}=\gamma^{Y^0,Y^1} \text{ everywhere except a }\P\text{--null set}. \] This follows from the fact that $\gamma^{\cdot,\cdot}$ has a closed graph \cite[Lemma 3.11]{Ambrosio2013users} and from Aumann's selection theorem \cite[Theorem 6.9.13]{bogachev2007measure}.
    
    \item If $\Xc$ is a Hilbert space, then
    \[ G(\Xc)=\{ f : [0,1] \longrightarrow \Xc \, \vert\, f(t)=(1-t)x+ty,\; t \in [0,1] , \text{ for some } x,y \in \Xc\}, \] and so it holds $\P$--almost surely for every $t \in [0,T]$ that \[ Y^t = (1-t) Y^0+tY^1.\]
\end{enumerate}
\end{remark}

\begin{proof}[Proof of \Cref{lem:Benamou_Brenier_1}] Let $(\X^u)_{u \in [0,1]}$ be a constant speed geodesic in $\Fc\Pc_p.$ Then $u \mapsto \X^u$ is absolutely continuous in $\Fc\Pc_p$ and thus, invoking \Cref{prop:AC_curve_in_FP}, there exists a probability space $(\Omega,\Fc,(\Fc_t)_{t=1}^T,\P)$ with $(\Fc_t)_{t=1}^T$-adapted processes $(Y^u)_{u \in [0,1]}$ satisfying \Cref{item:BB1}. Since $\Y^u \sim \X^u,$ we have that $u \mapsto \Y^u$ is a constant speed geodesic and $|\X^\prime|(u)=|\Y^\prime|(u)$ for almost every $u \in [0,1].$ Consequently, as in the proof of \Cref{prop:AC1}, we have
\begin{align}
\begin{split}
  \int_{C(\Xc)} & d^p_{\cX,p}(f(0),f(1)) \eta(\mathrm{d}f) \\
  &=\int d^p_{\cX,p}(Y^0(\omega),Y^1(\omega)) \P(\mathrm{d}\omega) \geq \Ac\Wc_p^p(\Y^0,\Y^1)  \\
  &=\int_0^1 |\Y^\prime|^p(u)\mathrm{d}u
  =\int_0^1 \int_{C(\Xc)} |f^\prime|^p(u) \eta(\mathrm{d}f)\mathrm{d}u\\
  &= \int_{C(\Xc)} \int_0^1 |f^\prime|^p(u) \mathrm{d}u\,\eta(\mathrm{d}f),
  \end{split} \label{eqn:BB_est}
\end{align} 
where it is used that $u \mapsto \Y^u$ is a constant-speed geodesic in the third (in)equality. As the inequality $d^p_{\cX,p}(f(0),f(1))\leq\int_0^1 |f^\prime|^p(u)\mathrm{d}u$ is immediate, we deduce from (4.2) that
\[
d^p_{\cX,p}(f(0),f(1))=\int_0^1 |f^\prime|^p(u)\mathrm{d}u,\quad \eta\text{--a.s.}
\] 
and $\eta$ is thus concentrated on $G(\Xc).$ This proves \Cref{item:BB2}. \Cref{item:BB3} then immediately follows from \eqref{eqn:BB_est}.

Conversely, assume that $(\Omega,\Fc,(\Fc_t)_{t=1}^T,\P)$ and $(Y^u)_{u \in [0,1]}$ satisfy \Cref{item:BB1,item:BB2,item:BB3}. Then we have
\begin{align*} 
\Ac\Wc_p^p(\X^t,\X^s)=\Ac\Wc_p^p(\Y^t,\Y^s) &\leq \int d^p_{\cX,p}(Y^t(\omega),Y^s(\omega)) \P(\mathrm{d}\omega) \\
&=\int_{C(\Xc)} d^p_{\cX,p}(f(t),f(s)) \eta(\mathrm{d}f)\\
& =|t-s|^p \int_{C(\Xc)} d^p_{\cX,p}(f(0),f(1)) \eta(\mathrm{d}f) \\
&=|t-s|^p\int d^p_{\cX,p}(Y^0(\omega),Y^1(\omega)) \P(\mathrm{d}\omega) \\
&=|t-s|^p \Ac\Wc_p^p(\Y^0,\Y^1)=|t-s|^p \Ac\Wc_p^p(\X^0,\X^1),
\end{align*} 
where we used that  $({\rm id},{\rm id})_{\#}\P \in \cplbc(\Y^t,\Y^s),$ $\eta$ is concentrated on $G(\Xc)$, and \Cref{item:BB3}. Assume that the inequality above is strict for some $t<s.$ Then, using the triangle inequality, we have
\begin{align*}
    \Ac\Wc_p(\X^0,\X^1)&\leq\Ac\Wc_p(\X^0,\X^t)+\Ac\Wc_p(\X^t,\X^s)+\Ac\Wc_p(\X^s,\X^1) \\
    &< (t+(s-t)+(1-s))\Ac\Wc_p(\X^0,\X^1)=\Ac\Wc_p(\X^0,\X^1).
    \end{align*} This is a contradiction and we thus have
    \[\Ac\Wc_p^p(\X^t,\X^s)=|t-s|^p \Ac\Wc_p^p(\X^0,\X^1), \] 
    which concludes the proof.
\end{proof}

Finally, using our representation, we are able to formulate the adapted optimal transport problem as an energy minimization. We thus conclude this section with the following adapted version of the Benamou--Brenier theorem.

\begin{theorem}\label{thm:Benamou-Brenier}
   Suppose $\cX$ is a geodesic space, and let $(\X^0,\X^1) \in \cF\cP_p \times \cF\cP_p$ for some $p \in (0,\infty)$. Then
    \begin{equation*}
        \cA\cW^p_p(\X^0,\X^1) = \min\bigg\{\int_0^1\int_{C(\cX)} |f^\prime|^p(r)\eta(\d f)\d r \,\bigg|\, \eta \in \sA(\X^0,\X^1)\bigg\},
    \end{equation*}
 where $\sA(\X^0,\X^1)$ is the set of all measures $\eta \in \cP(C(\cX))$ concentrated on $AC^p(\cX)$ and such that $\eta= \law_{\P}((Y^u)_{u \in [0,1]}),$ for some $(\Omega, \Fc, (\Fc_t)_{t=1}^T,\P)$ supporting $(\Fc_t)_{t=1}^T$-adapted processes $Y^u$ with $\X^0 \sim \Y^0$ and $\X^1 \sim \Y^1,$ where $\Y^0,\Y^1$ are defined as in {\rm\Cref{lem:Benamou_Brenier_1}}.
\end{theorem}
\begin{proof}
    The inequality ``$\leq$'' follows from \Cref{prop:AC1}. Indeed, for each $\eta \in \sA(\X,\Y)$, 
    \begin{align*} 
        \cA\cW^p_p(\X^0,\X^1) 
        &= \cA\cW^p_p(\Y^0,\Y^1) 
        = \bigg(\int_0^1 |\Y^\prime|(r)\d r\bigg)^p \leq \int_0^1 |\Y^\prime|^p(r) \d r  \\
        &        \leq \int_0^1 \int_{C(\Xc)} |f^\prime |^p(u)\eta(\mathrm{d}f) \d r.
    \end{align*}
    Conversely, since $\cF\cP_p$ is a geodesic space, let $u \mapsto \Y^u$ be a constant speed geodesic such that $\Y^0 = \X^0$ and $\Y^1 = \X^1$. We suppose, without loss of generality, that $(\Y^u)_{u \in [0,1]}$ is constructed as in \Cref{prop:AC_curve_in_FP}. Then
    \begin{equation*}
        \cA\cW^p_p(\X^0,\X^1) 
        = \cA\cW^p_p(\Y^0,\Y^1)  = \int_0^1 |\Y^\prime|^p_p(r)\d r = \int_0^1 \int_{C(\Xc)} |f^\prime |^p_p(u)\eta(\mathrm{d}f) \d r.
    \end{equation*}
    Here, the last equality follows from \eqref{eqn:lisini_derivative_FP}. This concludes the proof.
\end{proof}

\section{Application: Skorokhod representation} \label{sec:skorokhod}

The standard Skorokhod representation theorem states that, if a sequence of probability measures converges weakly, then there exists a probability space on which one can define random variables with these distributions that converge almost surely. In essence, it allows us to replace weak convergence with almost sure convergence by constructing all the variables on a suitably chosen common probability space.

In this section, as an application of our results, we prove the Skorokhod representation of convergent sequences in the adapted Wasserstein space. Albeit in a different setting, the first result on the Skorokhod representation for convergent sequences in the adapted weak topology was proved by \citeauthor*{hoover1991convergence} \cite{hoover1991convergence}, who specifically considered random variables equipped with filtrations rather than processes. More recently, in the same setting as this paper, a Skorokhod representation for convergent sequences in the adapted Wasserstein distance was established by \citeauthor*{beiglbock2024probabilistic} \cite{beiglbock2024probabilistic}, using the so-called transfer principle.

The idea of our approach is simple: we linearly interpolate between elements of the sequence to obtain a curve, to which \Cref{prop:AC_curve_in_FP} can be applied. However, to ensure that the resulting curve is absolutely continuous, we must impose additional assumptions on the sequence.

\begin{definition} Let $p \in (1,\infty)$ and $(y^n)_{n \in \N}$ be a sequence in a metric space $(\cY,d_\cY)$. We say that $(y^n)_{n \in \N}$ has finite weighted $p$-variation if there exists a sequence of positive numbers $(b_n)_{n \in \N}$ such that
\[ \sum_{n=1}^\infty b_n=1 \quad {\rm and}\quad \sum_{n=1}^\infty \frac{1}{b_n^{p-1} }d^p_\cY(y^n,y^{n+1})<\infty.\]
\end{definition}

\begin{remark} Note that, if $(\X^n)_{n \in \N}$ satisfies \begin{align}\label{eqn:summable_diff}\sum_{n=1}^\infty\Ac\Wc_p(\X^n,\X^{n+1})<\infty,
\end{align} then it has finite weighted $p$-variation. Indeed, defining $\tilde b_n\coloneqq \Ac\Wc_p(\X^n,\X^{n+1})$ gives
\[ \sum_{n=1}^\infty \tilde b_n<\infty\quad {\rm and}\quad \sum_{n=1}^\infty \frac{1}{\tilde b_n^{p-1} }\Ac\Wc_p^p(\X^n,\X^{n+1})=\sum_{n=1}^\infty \tilde b_n<\infty.\] Thus, we may take 
\[ b_n \coloneqq \frac{\tilde b_n}{\sum_{i=1}^\infty \tilde b_i},\]
if $\sum_{i=1}^\infty \tilde b_i>0,$ while the statement is trivial if $\sum_{i=1}^\infty \tilde b_i=0.$
\end{remark}

The main result of this section follows. We show the Skorokhod representation of convergent sequences with finite weighted $p$-variation in the adapted Wasserstein space. We point out that the assumption on the speed of convergence $Y^n \longrightarrow Y^\infty$ is generally not necessary, see \cite[Corollary 7.5]{beiglbock2024probabilistic}. However, in order to apply our results, we need impose it to ensure the absolute continuity of the linearly interpolated curve. On the other hand, the resulting sequence almost surely inherits the same rate of convergence.

\begin{proposition} \label{prop:Skoro_repr} Let $p \in (1,\infty)$ and let $(\X^n)_{n \in \N\cup\{\infty\}}\subseteq \FP_p$. Then $(\X^n)_{n \in \N}$ has finite weighted $p$-variation and converges to $\X^\infty$ if and only if $(\X^n)_{n \in \N \cup \{ \infty \}}$ admits a Skorokhod representation, i.e., there exists a filtered probability space $(\Omega,\cF,(\cF_t)_{t=1}^T,\P)$ and a family of process $Y^n = (Y^n_t)_{t=1}^T : \Omega \rightarrow \cX$, $n \in \N \cup \{\infty\},$ such that the filtered process
    \begin{equation} \label{eqn:def_FP_Skorokhod}
        \Y^n = (\Omega,\cF,(\cF_t)_{t=1}^T,\P, (Y^n_t)_{t=1}^T)
    \end{equation}
    satisfies $\cA\cW_p(\Y^n,\X^n) = 0$ for every $n \in \N \cup \{\infty\}$, and the sequence $(Y^n)_{n \in \N}$ converges $\P$--a.s.\ to $Y^\infty$ and has finite weighted $p$-variation in $L^p(\P;\Xc).$
\end{proposition}

\begin{remark}
\begin{enumerate}[label=(\roman*)]
\item Here, $L^p(\P;\Xc)$ denotes the space of all random variables $Y:\Omega \rightarrow \Xc$ with $\E^\P\big[ d^p_{\cX,p}(Y,x_0)\big]<\infty$ for some (and thus all) $x_0 \in \Xc$ endowed with the metric
    \[
        d_{L^p(\P;\cX)}(Y,Y^\prime) \coloneqq \big(\E^\P[d^p_{\cX,p}(Y,Y^\prime)]\big)^{1/p}.
    \]
    \item More specifically, $(Y^n)_{n \in \N}$ having finite weighted $p$-variation in $L^p(\P;\Xc)$ means that there exist positive numbers $(b_n)_{n \in \N}$ such that
    \[
        \sum_{n = 1}^\infty b_n = 1 \quad \textnormal{and} \quad \sum_{n = 1}^\infty \frac{1}{b^{p-1}_n} \E^\P\big[ d^p_{\cX,p}(Y^n,Y^{n+1})\big] < \infty,
    \]
    which, in particular, implies
    \[
      \sum_{n = 1}^\infty \frac{1}{b^{p-1}_n} d^p_{\cX,p}(Y^n,Y^{n+1}) < \infty, \; \textnormal{$\P$--a.s.}
    \]
\end{enumerate}
\end{remark}

\begin{proof} Let us first prove the ``only if'' implication. Set $u_0=0$ and $u_n=\sum_{i=1}^n b_n$ for $n \in \N,$ and note that $u_n \in [0,1)$ for every $n \in \N\cup\{0\}$. We then define a curve in $\FP_p$ as follows: We set $\tilde \X^{u_n}\coloneqq \X^n$ for $n \in \N$, and for $u \in (u_i,u_{i+1})$ we define $\tilde \X^u$ using the displacement interpolation as in the proof of \Cref{prop:AC_curve_in_FP}. This defines the curve $(\tilde \X^u)_{u \in [0,1)}.$ Finally, we set $\tilde \X^1=\X^\infty.$ As $(\X^n)_{n \in \N}$ has finite weighted $p$-variation, it follows from direct calculations that $(\tilde \X^u)_{u \in [0,1]}$ is a curve in $AC^p(\FP_p).$ The statement then follows  by invoking \Cref{prop:AC_curve_in_FP}.

The ``if'' implication is immediate. Assume that there exists a filtered probability space $(\Omega,\cF,(\cF_t)_{t=1}^T,\P)$ and a family of process $Y^n = (Y^n_t)_{t=1}^T : \Omega \rightarrow \cX$, $n \in \N \cup \{\infty\},$ such that the filtered process
    \begin{equation*}
        \Y^n = (\Omega,\cF,(\cF_t)_{t=1}^T,\P, (Y^n_t)_{t=1}^T)
    \end{equation*}
    satisfies $\cA\cW_p(\Y^n,\X^n) = 0$ for every $n \in \N \cup \{\infty\}$, the sequence $(Y^n)_{n \in \N}$ converges $\P$--a.s.\ to $Y^\infty$ and has finite weighted $p$-variation in $L^p(\P;\Xc)$. Then, 
\[\Ac\Wc_p^p(\X^n,\X^{m})=\Ac\Wc_p^p(\Y^n,\Y^{m})\leq \E^{\P}\big[d_{\Xc,p}^p(Y^n,Y^{m})\big],\quad m,n \in \N\cup\{\infty\},\] which immediately gives that the weighted $p$-variation of $(\X^n)_{n \in\N}$ is finite. Moreover, it is easy to verify that $Y^n \rightarrow Y^\infty$ in $L^p(\P;\Xc)$, which gives
\[\Ac\Wc_p^p(\X^n,\X^\infty)\leq \E^{\P} [d_{\Xc,p}^p(Y^n,Y^\infty)]\xrightarrow{n \rightarrow \infty} 0.\]
The proof is thus concluded.
\end{proof}

As a corollary, we obtain a characterization of convergent sequences in the adapted Wasserstein space.

\begin{corollary} Let $(\X^n)_{n \in \N \cup \{\infty\}}$ be a family in $\FP_p.$ Then, $\lim_{n \rightarrow \infty} \Ac\Wc_p(\X^n,\X^\infty)=0$ holds if and only if every subsequence of $(\X^n)_{n \in \N \cup \{\infty\}}$ contains a further subsequence, say $(\X^{n_l})_{l \in \N},$ such that $(\X^{n_l})_{l \in \N} \cup \{\X^\infty\}$ admits a Skorokhod representation.
\end{corollary}
\begin{proof} Let $(\X^n)_{n \in \N \cup \{\infty\}}$be such that $\lim_{n \rightarrow \infty} \Ac\Wc_p(\X^n,\X^\infty)=0.$ From every subsequence, we can select a further subsequence for which \eqref{eqn:summable_diff} holds, so that \Cref{prop:Skoro_repr} can be applied.

Conversely, assume that $\lim_{n \rightarrow \infty} \Ac\Wc_p(\X^n,\X^\infty)=0$ does not hold. Then, we can find a subsequence, say $(\X^{n_l})_{l \in \N}$, and $\varepsilon>0$, such that $\Ac\Wc_p(\X^{n_l},\X^\infty)>\varepsilon$ for every $l \in \N.$ It follows that $(\X^{n_l})_{l \in \N}$ cannot contain a subsequence that admits a Skorokhod representation.
\end{proof}

\section{Proof of \texorpdfstring{\Cref{prop:AC_curve_in_FP}}{}} \label{sec:proof_repres}

\begin{proof}[Proof of \Cref{prop:AC_curve_in_FP}]
    The main strategy of the proof is defining a sequence of filtered processes derived from a piece-wise interpolation of the absolutely continuous curve. This sequence of finer and finer approximations then turns out to be precompact in the space of filtered processes $\FP_p$. Subsequently, we show that a transformation of an accumulation point of this sequence satisfies the desired properties.
    
    We start with the interpolation of the curve. Let
    \begin{equation*}
        \X^u = (\Omega^u,\cF^u,(\cF^u_t)_{t = 1}^T,\P^u,(X^u_t)_{t = 1}^T), \; u \in [0,1].
    \end{equation*}
    We suppose without loss of generality that each $\Omega^u$ is a Polish space, as otherwise we would replace every $\X^u$ with its canonical representative, see \Cref{sec::canonical_filtered_process}.
    For $N \in \N$, let
    \begin{equation*}
        \Omega^{[N]} \coloneqq  \prod_{i = 0}^{2^N} \Omega^{i/2^N}= \Omega^{0} \times \Omega^{1/2^N} \times \cdots \times \Omega^{(2^N-1)/2^N} \times \Omega^{1}.
    \end{equation*}
    We denote by $\pi^{i/2^N}: \Omega^{[N]}\rightarrow \Omega^{i/2^N}$ the projection and by $\gamma^i_N \in \cplbc(\X^{i/2^N},\X^{(i+1)/2^N})$ an optimal bicausal coupling for the transport problem $\Ac\Wc_p(\X^{i/2^N},\X^{(i+1)/2^N}).$\footnote{Given that we are working with the canonical representative, the existence of an optimal bicausal coupling can be deduced from \cite[Lemma A.3]{eckstein2023computational} via the standard construction.} Using \cite[Lemma A.5]{AcKrPa23a}, there exists $\gamma_N \in \cP(\Omega^{[N]})$ such that 
    \[
        \gamma_N \in \cplmc(\X^{0},\X^{1/2^N},\ldots,\X^{1}) \; \textnormal{and} \; (\pi^{i/2^N},\pi^{(i+1)/2^N})_{\#}\gamma_N = \gamma^i_N
    \]
    for each $i \in \{0,1,\ldots,2^N-1\}$. Let 
    \[
        \cF^{[N]} = \otimes_{i = 0}^{2^N}\cF^{i/2^N}, \; \cF^{[N]}_t = \otimes_{i = 0}^{2^N}\cF_t^{i/2^N}, \; t = 1,\ldots,T,
    \]
    and let
    \[
        \sigma_N:\Omega^{[N]} \longrightarrow L^p(\cX)
    \]
    be such that for $\omega \in \Omega^{[N]}$, the map $\sigma_N(\omega) : [0,1] \longrightarrow \cX$ is given by
    \begin{equation*}
        \sigma^u_N(\omega) \coloneqq
        \begin{cases}
            \big(X^{1}_1(\pi^{1}(\omega)),\ldots,X^{1}_T(\pi^{1}(\omega))\big), & \textnormal{if $u = 1$}, \\[0.5em]
            \big(X^{i/2^N}_1(\pi^{i/2^N}(\omega)),\ldots,X^{i/2^N}_T(\pi^{i/2^N}(\omega))\big), & \textnormal{if $u \in [i/2^N,(i+1)/2^N)$}.
        \end{cases}
    \end{equation*}
    We can also see $\sigma_N$ as a process $\sigma_N=(\sigma_{N,t})_{t=1}^T : \Omega^{[N]} \rightarrow \prod_{t=1}^T L^p(\Xc_t)$ with
    \begin{align*}
        \sigma^u_{N,t}(\omega) \coloneqq 
        \begin{cases}
            X^1_t(\pi^1(\omega)), & \textnormal{if $u = 1$}, \\[0.5em]
            X^{i/2^N}_t\big(\pi^{i/2^N}(\omega)\big), & \textnormal{if $u \in [i/2^N,(i+1)/2^N)$.}
        \end{cases}
    \end{align*} 
    We then consider the sequence of filtered processes
    \begin{align*}
        \Y^N &\coloneqq(\Omega^{[N]},\cF^{[N]},(\cF^{[N]}_t)_{t = 1}^T,\gamma_N,(\sigma^N_t)_{t = 1}^T) \in \FP( \Pi_{t = 1}^TL^p(\cX_t)), \quad N \in \N, \\
        \Y^{N,u} &\coloneqq(\Omega^{[N]},\cF^{[N]},(\cF^{[N]}_t)_{t = 1}^T,\gamma_N, Y^{N,u}) \in \FP(\cX), \quad N\in \N,\; u \in [0,1],
    \end{align*} 
    where $Y^{N,u}\coloneqq e^u \circ \sigma_{N}$, with $e^u(f) = f(u)$ denoting the evaluation of a map $f:[0,1] \rightarrow \cX$ at $u \in [0,1]$. Note that  $\gamma_N \in \cplmc(\X^{0},\X^{1/2^N},\ldots,\X^{1})$ implies  $(\pi^{i/2^N},{\rm id})_{\#}\gamma_N \in \cplbc(\X^{i/2^N},\Y^{N,i/2^N})$.\footnote{Although the argument is tedious, this follows from the definition of multicausality when considering both $\X^{i/2^N}$ and $\Y^{N,i/2^N}$ as defined on $\Omega^{[N]}$.} Consequently,
    \begin{align}
        \Ac\Wc_p^p(\X^{i/2^N},\Y^{N,i/2^N})
        &\leq  \int d^p_{\cX,p}\big(X^{i/2^N}\big(\pi^{i/2^N}(\omega)\big),Y^{N,i/2^N}(\omega)\big) \gamma_N(\mathrm{d}\omega) \nonumber \\
        &=  \int d^p_{\cX,p}\big(X^{i/2^N}\big(\pi^{i/2^N}(\omega)\big),e^{i/2^N} \circ \sigma^N(\omega)\big) \gamma_N(\mathrm{d}\omega) \nonumber \\
        &= \int d^p_{\cX,p}(e^{i/2^N} \circ \sigma^N(\omega),e^{i/2^N} \circ \sigma^N(\omega)) \gamma_N(\mathrm{d}\omega) \nonumber \\
        &=0 \label{eqn:AW_is_0}
    \end{align}
    holds for each $i \in \{1,\ldots,2^N\}$.
    
    We now show that the collection of filtered processes $(\Y^N)_{N \in \N}$, or more precisely, of the corresponding equivalence classes, is precompact in $\FP( L^p(\Xc))$ with respect to the adapted weak topology. To do this, we identify the isometric metric spaces $\Pi_{t = 1}^T L^p(\cX_t)$ and $L^p(\cX)$.\footnote{We endow $\Pi_{t = 1}^T L^p(\cX_t)$ with the $p$ product metric.} By \Cref{lem::relatively_compact} and Prokhorov's theorem,\footnote{Here, we are interested in relative compactness with respect to the adapted weak topology, that is, the topology generated by replacing each $d_{\cX_\smalltext{t}}$ with $d_{\cX_\smalltext{t}} \land 1$. Uniform $p$-integrability, as additionally required in \Cref{lem::relatively_compact}, is then automatically satisfied.} this is equivalent to showing tightness of the family of probability measures
    \begin{equation*}
        \eta_N \coloneqq (\sigma_N)_\# \gamma_N \in \cP(L^p(\cX)), \; N \in \N.
    \end{equation*}
    To simplify the notation, we write $X^{i/2^N}(\omega)$ instead of $X^{i/2^N}(\pi^{i/2^N}(\omega))$ in what follows, and thus identify the random variable $X^{i/2^N}$ on $\Omega^{i/2^N}$ with the random variable $X^{i/2^N} \circ \pi^{i/2^N}$ on $\Omega^{[N]}$, whenever no confusion arises.

    \medskip {\bf Step 1: Tightness.} Since $u \mapsto \X^u$ is continuous with respect to $\Ac\Wc_p$ in $\FP_p$, it follows from the definition of the adapted Wasserstein distance that the map $u \mapsto \law_{\P^u} (X^u) \eqqcolon \mu^u$ is continuous with respect to weak topology on $\Pc(\Xc).$ Thus $\mathscr{F}\coloneqq\{ \mu^u : u \in [0,1] \}$ is compact in $\cP(\Xc)$ and therefore tight by Prokhorov's theorem. In particular, $\mathscr{F}$ is bounded, that is, for a fixed point $\bar{x} \in \Xc$, there exists $C_1>0$ such that
    \[ 
        \int_{\Xc} d^p_{\cX,p}(x,\bar{x}) \mu^u(\mathrm{d}x)=\Wc_p^p(\mu^u,\delta_{\bar{x}})\leq C_1.
    \] 
    Moreover (see \cite[Ch.~8, Ex.~8.6.5, p.~205]{bogachev2007measure}), since $\mathscr{F}$ is tight, there exists a function $\psi: \Xc \rightarrow [0,\infty]$ whose sublevels $\{x \in \Xc \,|\, \psi(x) \leq c \}$ are compact in $\Xc$ for any $c \in [0,\infty)$ and such that
    \[ 
        C_2\coloneqq \sup_{u \in [0,1]} \int \psi(x) \mu^u(\mathrm{d}x)<\infty.
    \] 
    We fix $\bar{x} \in \cX$ and, for $f \in L^p(\Xc)$, we set
    \[ 
        \Phi(f)\coloneqq\int_0^1 d^p_{\cX,p}(f(u),\bar{x})\mathrm{d}u+\int_0^1 \psi(f(u))\mathrm{d}u+\sup_{0<h<1} \int_0^{1-h}\frac{d^p_{\cX,p}(f(u+h),f(u))}{h}\mathrm{d}u.
    \] 
    
    Then, by an application of Tonelli's theorem,
    \begin{align*}
        \int_{L^p(\Xc)}& \int_0^1 \big(d^p_{\cX,p}(f(u),\bar{x})+\psi(f(u))\big) \mathrm{d}u\,\eta_N(\mathrm{d}f) \\
        &=\int_0^1 \int_{\Omega^{[N]}} \big(d^p_{\cX,p}(\sigma^u_N(\omega),\bar{x})+\psi(\sigma^u_N(\omega))\big) \gamma_N(\mathrm{d}\omega)\mathrm{d}u\nonumber\\
        &= \sum_{i=0}^{2^N-1} \int_{i/2^N}^{(i+1)/2^N} \int_{\Xc} (d^p_{\cX,p}(x,\bar{x})+\psi(x)) \mu^{i/2^N}(\mathrm{d}x)\mathrm{d}u \nonumber\\
        &= \frac{1}{2^N} \sum_{i=0}^{2^N-1} \int_{\Xc} ( d^p_{\cX,p}(x,\bar{x})+\psi(x)) \mu^{i/2^N}(\mathrm{d}x) \nonumber\\
        &\leq \frac{1}{2^N} \sum_{i=0}^{2^N-1} (C_1+C_2) \nonumber\\
        &=C_1+C_2.
    \end{align*}
Furthermore, if $h < 1/2^N,$ then
    \begin{align*} 
        \int_0^{1-h} d^p_{\cX,p}(\sigma^{u+h}_N(\omega),\sigma^{u}_N(\omega)) \mathrm{d}u&=\sum_{i=0}^{2^N-2} \int_{i/2^N}^{(i+1)/2^N} d^p_{\cX,p}(\sigma^{u+h}_N(\omega),\sigma^{u}_N(\omega)) \mathrm{d}u \nonumber\\
        &=h \sum_{i=1}^{2^N-2} d^p_{\cX,p}\big(X^{i/2^N}(\omega),X^{(i+1)/2^N}(\omega)\big),
    \end{align*} 
    because $\sigma^{u+h}_N(\omega)=\sigma^{u}_N(\omega)$ for $u\in [i/2^N,(i+1)/2^N-h).$ In case $1/2^N \leq h <1$, we can find $k \in \N$ such that $k/2^N\leq h < (k+1)/2^N.$ Then, using the triangle inequality, we obtain 
    \[ 
    d_{\cX,p}(\sigma^{u+h}_N(\omega),\sigma^{u}_N(\omega))\leq \sum_{i=0}^kd_{\cX,p}\big(\sigma^{u+(i+1)/2^N}_N(\omega),\sigma^{u+i/2^N}_N(\omega)\big),
    \] 
    which by H\"older inequality gives
    \[ d^p_{\cX,p}(\sigma^{u+h}_N(\omega),\sigma^{u}_N(\omega))\leq (k+1)^{p-1} \sum_{i=0}^kd^p_{\cX,p}\big(\sigma^{u+(i+1)/2^N}_N(\omega),\sigma^{u+i/2^N}_N(\omega)\big).\] Consequently,
    \begin{align}
        \int_0^{1-h} &d^p_{\cX,p}(\sigma^{u+h}_N(\omega),\sigma^{u}_N(\omega)) \mathrm{d}u \nonumber\\
        &\leq \int_0^{1-k/2^N} d^p_{\cX,p}(\sigma^{u+h}_N(\omega),\sigma^{u}_N(\omega)) \mathrm{d}u \nonumber \\
        &\leq \int_0^{1-k/2^N} (k+1)^{p-1} \sum_{i=0}^kd^p_{\cX,p}(\sigma^{u+(i+1)/2^N}_N(\omega),\sigma^{u+i/2^N}_N(\omega)) \mathrm{d}u \nonumber \\
        &\leq (k+1)^{p-1} \sum_{i=0}^k \frac{1}{2^N} \sum_{j=0}^{2^N-k-1} d^p_{\cX,p}(X^{(i+j+1)/2^N}(\omega),X^{(i+j)/2^N}(\omega)) \label{eqn:tightness_Lisini1} \\
        &\leq \frac{(k+1)^{p}}{2^N} \sum_{j=0}^{2^N-1}d^p_{\cX,p}(X^{(j+1)/2^N}(\omega),X^{j/2^N}(\omega)), \label{eq:concentrated_on_w_1_p}
    \end{align} where the last inequality follows from the fact that, in \eqref{eqn:tightness_Lisini1}, the term \linebreak $d^p_{\cX,p}(X^{(i+j+1)/2^N}(\omega),X^{(i+j)/2^N}(\omega))$ is counted at most $k+1$ times. The inequality $k/2^N\leq h$ further gives 
    \begin{equation}\label{eqn:lisini_bound_for_w_1_p}
    \int_0^{1-h} d^p_{\cX,p}(\sigma^{u+h}_N(\omega),\sigma^{u}_N(\omega)) \mathrm{d}u \leq h \frac{(k+1)^p}{k} \sum_{j=0}^{2^N-1}d^p_{\cX,p}(X^{(j+1)/2^N}(\omega),X^{j/2^N}(\omega)).
    \end{equation}
    We note that the function $a \mapsto (a+1)^p/a$ is decreasing on $(0,1/(p-1))$ and increasing on $(1/(p-1),\infty).$ Thus, since $1 \leq k \leq 2^N-1,$ we have
    \[ \frac{(k+1)^p}{k}\leq \begin{cases} 2^p, & \text{if } k \leq 1/(p-1), \\ \frac{2^{Np}}{2^N-1}, & \text{if } k > 1/(p-1).  \end{cases}\]
    This together with the bound \eqref{eqn:lisini_bound_for_w_1_p} yields
    \begin{multline} \label{eqn:lisini_bound2}
        \sup_{h \in (0,1)} \int_0^{1-h} \frac{d^p_{\cX,p}(\sigma^{u+h}_N(\omega),\sigma^{u}_N(\omega))}{h} \mathrm{d}u \\
        \leq \Big(2^p+ \frac{2^{Np}}{2^N-1} \Big) \sum_{j=0}^{2^N-1}d^p_{\cX,p}(X^{(j+1)/2^N}(\omega),X^{j/2^N}(\omega)).
    \end{multline}
From \begin{align}\label{eqn:lisini_int_gamma_n_bound}
        \int_{\Omega^{[N]}} \sum_{i=0}^{2^N-1} d^p_{\cX,p}(X^{i/2^N}(\omega),X^{(i+1)/2^N}(\omega)) \gamma_N(\mathrm{d}\omega) 
        &=\sum_{i=0}^{2^N-1} \Ac\Wc_p^p(\X^{i/2^N},\X^{(i+1)/2^N}) \nonumber\\
        &\leq \sum_{i=0}^{2^N-1} \bigg( \int_{i/2^N}^{(i+1)/2^N} |\X^\prime|(u) \mathrm{d}u \bigg)^p \nonumber\\
        &\leq \frac{1}{2^{N(p-1)}} \sum_{i=0}^{2^N-1} \int_{i/2^N}^{(i+1)/2^N} |\X^\prime|^p(u) \mathrm{d}u  \nonumber\\
        &\leq \frac{1}{2^{N(p-1)}} \int_0^1 |\X^\prime|^p(u) \mathrm{d}u,
    \end{align} 
  together with \eqref{eqn:lisini_bound2}, we find that 
    \[ \int_{L^p(\Xc)} \sup_{h \in (0,1)} \int_0^{1-h} \frac{d^p_{\cX,p}(f(u+h),f(u))}{h} \mathrm{d}u \,\eta_N(\mathrm{d}f) \leq (2^p+2) \int_0^1 |\X^\prime|^p(u) \mathrm{d}u.  \]
    Because the curve $u \mapsto \X^u$ has finite $p$-energy, the right-hand side is finite, and we conclude that
    \begin{multline*}
    \sup_{N \in \N}\bigg\{\int_{L^p(\Xc)} \bigg( \int_0^1 d^p_{\cX,p}(f(u),\bar{x})+\psi(f(u)) \mathrm{d}u\\
    +\sup_{h \in (0,1)} \int_0^{1-h} \frac{d^p_{\cX,p}(f(u+h),f(u))}{h} \mathrm{d}u\bigg) \eta_N(\mathrm{d}f)\bigg\}
    < \infty.
    \end{multline*}
  Further, the function
  \begin{align*} L^p(\Xc)  &\longrightarrow [0,\infty] \\
  f &\longmapsto \int_0^1 d^p_{\cX,p}(f(u),\bar{x})
  +\psi(f(u))\mathrm{d}u+\sup_{h \in (0,1)} \int_0^{1-h} \frac{d^p_{\cX,p}(f(u+h),f(u))}{h} \mathrm{d}u
  \end{align*}
   has compact sublevel sets; see \cite[Theorem 2]{lisini2007characterization} and the references therein. Thus, using \cite[Example 8.6.5]{bogachev2007measure}, we have that the sequence $(\eta_N)_{N \in \N}$ is tight in $\cP(L^p(\Xc)).$ Consequently, the collection $(\Y^N)_{N \in \N}$ is precompact in $\FP(L^p(\Xc))$ with respect to the adapted weak topology and thus admits a cluster point 
    \[ \Y=( \Omega, \Fc,(\Fc_t)_{t=1}^T,\P, \sigma ).\] We write $\eta\coloneqq \law_{\P} (\sigma) \in \cP(L^p(\Xc))$ for the law of $\sigma$ under $\P$. It is straightforward to check that $\eta$ is necessarily a cluster point of the sequence $(\eta_N)_{N \in \N}.$

    \medskip
    {\bf Step 2: $\eta$ is concentrated on $W^{1,p}(\cX)$.} Note that it suffices to show that
    \begin{equation}\label{eq:etaconc}
        \sup_{h \in (0,1)} \int_0^{1-h}\frac{d^p_{\cX,p}(f(u+h),f(u))}{h^p}\d u < \infty \;\, \text{for $\eta$--a.e.~$f \in L^p(\cX)$}.
    \end{equation}
    First note that the functions $\fg_N : L^p(\cX) \longrightarrow [0,\infty]$, $N \in \N$, defined by
    \begin{equation*}
        \fg_N(f) \coloneqq \sup_{h \in [1/2^N, 1)} \int_0^{1-h} \frac{d^p_{\cX,p}(f(u+h),f(u))}{h^p}\d u,
    \end{equation*}
    satisfy $\fg_N(f) \leq \fg_{N+1}(f)$ and 
    \begin{equation*}
        \sup_{N \in \N} \fg_N(f) = \sup_{h \in (0,1)} \int_0^{1-h}\frac{d^p_{\cX,p}(f(u+h),f(u))}{h^p}\d u.
    \end{equation*}
    Monotone convergence then yields
    \begin{equation*}
        \int_{L^p(\Xc)} \sup_{N \in \N}\fg_N(f)\eta(\d f) 
        = \sup_{N \in \N} \int_{L^p(\Xc)}  \fg_N(f)\eta(\d f).
    \end{equation*}
    For \eqref{eq:etaconc} to hold, we it is enough to find a constant $C \in [0,\infty)$ such that   \begin{equation}\label{eqn:lisini_reduction_concentrated_w_1_p}
        \int_{L^p(\Xc)}  \fg_N(f)\eta(\d f) \leq C, \; N \in \N.
    \end{equation}    
    Fix $N \in \N$, $h \in [1/2^N, 1)$ and $k \in \N$ such that $h \in [k/2^N, (k+1)/2^N)$. Then \eqref{eq:concentrated_on_w_1_p} together with $1/2^N \leq 2^{N(p-1)}h^p/k^p$  implies
    \begin{align*}
        \int_0^{1-h} d^p_{\cX,p}(\sigma^{u+h}_N(\omega),\sigma^{u}_N(\omega)) \mathrm{d}u 
        &\leq \frac{(k+1)^{p}}{2^N} \sum_{j=0}^{2^N-1}d^p_{\cX,p}(X^{(j+1)/2^N}(\omega),X^{j/2^N}(\omega)) \\
        &\leq h^p\frac{(k+1)^p}{k^p}2^{N(p-1)}\sum_{j = 0}^{2^N-1}d^p_{\cX,p}(X^{(j+1)/2^N}(\omega),X^{j/2^N}(\omega)) \\
        &\leq h^p 2^p 2^{N(p-1)}\sum_{j = 0}^{2^N-1}d^p_{\cX,p}(X^{(j+1)/2^N}(\omega),X^{j/2^N}(\omega)),
    \end{align*}
    where the last inequality follows from $(k+1)^p /k^p \leq 2^p$. A rearrangement of the terms then yields
    \begin{multline*}
        \sup_{h \in [1/2^N,1)}\int_0^{1-h} \frac{d^p_{\cX,p}(\sigma^{u+h}_N(\omega),\sigma^{u}_N(\omega))}{h^p} \mathrm{d}u 
        \\
        \leq 2^p 2^{N(p-1)}\sum_{j = 0}^{2^N-1}d^p_{\cX,p}(X^{(j+1)/2^N}(\omega),X^{j/2^N}(\omega)).
    \end{multline*}
    We now integrate both sides with respect to $\gamma_N$ and obtain
    \begin{align*}
        \int_{L^p(\Xc)}  \fg_N(f)\eta_N(\d f) 
        &\leq 2^p 2^{N(p-1)}\sum_{j = 0}^{2^N-1}\int_{\Omega^{[N]}} d^p_{\cX,p}(X^{(j+1)/2^N}(\omega),X^{j/2^N}(\omega)) \gamma_N(\d\omega) \\
        &\leq 2^p \int_0^1 |\X^\prime|^p(u) \mathrm{d}u \eqqcolon C,
    \end{align*}
    where we used \eqref{eqn:lisini_int_gamma_n_bound} in the second inequality. Since $(\fg_N)_{N \in \N}$ is a point-wise non-decreasing sequence of functions, it follows that
    \begin{equation}\label{eq:gNk}
        \int_{L^p(\Xc)} \fg_N(f)\eta_{N+k}(\d f) \leq \int_{L^p(\Xc)} \fg_{N+k}(f)\eta_{N+k}(\d f) \leq C, \;\, k,N \in \N.
    \end{equation}
    Additionally, each $\fg_N$ is lower semi-continuous, since, for each sequence $(f_n)_{n \in \N} \subseteq L^p(\cX)$ converging to $f \in L^p(\cX)$ with respect to $d_{L^p}$, we have 
    \begin{align*}
        \big(\fg_N(f)\big)^{1/p} &=\bigg( \sup_{h \in [1/2^N,1)} \int_0^{1-h} \frac{d^p_{\cX,p}(f(u+h),f(u))}{h^p}\d u \bigg)^{1/p} \\
        &\leq \bigg(\sup_{h \in [1/2^N,1)} \int_0^{1-h}\frac{d^p_{\cX,p}(f(u+h),f_n(u))}{h^p}\d u\bigg)^{1/p} \\
        &\quad + \bigg(\sup_{h \in [1/2^N,1)} \int_0^{1-h}\frac{d^p_{\cX,p}(f_n(u),f(u))}{h^p}\d u\bigg)^{1/p} \\
        &\leq \bigg(\sup_{h \in [1/2^N,1)} \int_0^{1-h}\frac{d^p_{\cX,p}(f(u+h),f_n(u))}{h^p}\d u\bigg)^{1/p} \\
        &\quad+ 2^{pN}\bigg(\sup_{h \in [1/2^N,1)} \int_0^{1-h}d^p_{\cX,p}(f_n(u),f(u))\d u\bigg)^{1/p} \\
        &\leq \bigg(\sup_{h \in [1/2^N,1)} \int_0^{1-h}\frac{d^p_{\cX,p}(f(u+h),f_n(u))}{h^p}\d u\bigg)^{1/p} + 2^{pN}\big(d^p_{L^p}(f_n,f)\big)^{1/p} \\
        &\leq \bigg(\sup_{h \in [1/2^N,1)} \int_0^{1-h}\frac{d^p_{\cX,p}(f(u+h),f_n(u+h))}{h^p}\d u\bigg)^{1/p} \\
        &\quad + \bigg(\sup_{h \in [1/2^N,1)} \int_0^{1-h}\frac{d^p_{\cX,p}(f_n(u+h),f_n(u))}{h^p}\d u\bigg)^{1/p} + 2^{pN}\big(d^p_{L^p}(f_n,f)\big)^{1/p} \\
        &= 2^{pN+1}\big( d^p_{L^p}(f_n,f)\big)^{1/p} + \big(\fg_N(f_n)\big)^{1/p}.
    \end{align*}
    Here, we used Minkowski's inequality in the second and and third-to-last inequalities.  Weak convergence of $\eta_{N+k}$ to $\eta$ as $k$ tends to infinity, together with lower semi-continuity of $\fg_N$ and \eqref{eq:gNk}, yields for every $N \in \N$:
    \begin{align*}
         \int_{L^p(\Xc)} \fg_N(f)\eta(\d f) 
         &\leq \liminf_{k \rightarrow\infty} \int_{L^p(\Xc)} \fg_{N}(f)\eta_{N+k}(\d f) 
         \\
         &\leq \liminf_{k \rightarrow\infty} \int_{L^p(\Xc)} \fg_{N+k}(f)\eta_{N+k}(\d f) \leq C.
    \end{align*}
    This implies \eqref{eqn:lisini_reduction_concentrated_w_1_p}, and thus $\eta$ is concentrated on $W^{1,p}(\cX)$.

    \medskip
    {\bf Step 3: Construction of $(\Y^u)_{u \in [0,1]}$ and $\tilde\eta$ concentrated on $AC(\cX)$.} Since $\sigma$ is $\P$--almost surely $W^{1,p}(\cX)$--valued, and since $W^{1,p}(\cX)$ is a Borel subset of $L^p(\cX)$, we suppose without loss of generality that $\sigma(\omega) \in W^{1,p}(\cX)$ for every $\omega \in \Omega$.\footnote{We would otherwise need to consider the filtered process restricted to a subset of $\Omega$.} We thus consider $\eta$ as a probability measure on the Borel-$\sigma$-algebra of $W^{1,p}(\cX)$. By \Cref{lem:sobolev_repres}, there exists a Borel-measurable map $T: W^{1,p}(\cX) \rightarrow C(\cX)$ mapping each equivalence class $[f]_{\sim} \in W^{1,p}(\cX)$ to its continuous representative $T([f]_{\sim}) \in C(\cX)$, that is, $f(u) = T([f]_{\sim})(u)$ for a.e.~$u \in [0,1]$. We then define 
    $Y^u : \Omega \rightarrow \cX$ by $Y^u(\omega)\coloneqq  (e^u\circ T \circ \sigma)(\omega)$ and write
    \[ 
        \Y^u\coloneqq( \Omega, \Fc,(\Fc_t)_{t=1}^T,\P, (Y^u_t)_{t = 1}^T ), \quad u \in [0,1].
    \]
    Then $\tilde\eta \coloneqq \law_{\P} ((Y^u)_{u \in [0,1]}) =  T_\# \eta \in \cP(C(\cX))$ is concentrated on $AC(\cX)$ since $T([f]_{\sim}) \in AC(\cX)$ for each $[f]_{\sim} \in W^{1,p}(\cX)$.

    \medskip {\bf Step 4: Proof of \eqref{eqn:lisini_derivative_FP}.} We only prove ``$\geq$'', as the converse inequality follows from \eqref{eqn:lisini_derivative_FP_leq}. By Tonelli's theorem and an application of Lebesgue's differentiation theorem, it suffices to show that, for each pair of dyadic numbers $0 < s_1 < s_2 < 1$, we have
    \begin{equation*}
        \int_{C(\cX)}\bigg(\int_{s_1}^{s_2} |f^\prime|(u)\d u\bigg) \tilde\eta(\d f) \leq \int_{s_1}^{s_2} |\X^\prime|(u)\d u.
    \end{equation*}
    We start by fixing dyadic numbers $s_1,s_2,h$ satisfying $0 < s_1 < s_2 < s_2 + h < 1$. We choose $N \in \N$ large enough such that $s_1 = s_{1,N}/2^N$, $s_2 = s_{2,N}/2^N$ and $h = h_N/2^N$ for some $s_{1,N},s_{2,N},h_N \in \N$. Then
    \begin{align*}
        &\int_{s_1}^{s_2} d^p_{\cX,p}(\sigma_N^{u+h}(\omega),\sigma_N^u(\omega))\d u \\
        &= \sum_{k = s_{1,N}}^{s_{2,N}-1}\int_{k/2^N}^{(k+1)/2^N} d^p_{\cX,p}(\sigma_N^{u+h}(\omega),\sigma_N^u(\omega))\d u \\
        &= \sum_{k = s_{1,N}}^{s_{2,N}-1} \int_{k/2^N}^{(k+1)/2^N} d^p_{\cX,p}(\sigma_N^{(k+h_N)/2^N}(\omega),\sigma_N^{k/2^N}(\omega))\d u \\
        &\leq\sum_{k = s_{1,N}}^{s_{2,N}-1} \bigg( \int_{k/2^N}^{(k+1)/2^N} \bigg( \sum_{j = 0}^{h_N-1} d_{\cX,p}(\sigma_N^{(k+h_N)/2^N}(\omega),\sigma_N^{k/2^N}(\omega))\bigg)^p\d u\bigg)  \\
        &\leq \sum_{k = s_{1,N}}^{s_{2,N}-1}\bigg( \sum_{j = 0}^{h_N-1} \bigg( \int_{k/2^N}^{(k+1)/2^N} d^p_{\cX,p}(\sigma_N^{(k+j+1)/2^N}(\omega),\sigma_N^{(k+j)/2^N}(\omega))\d u\bigg)^{1/p}\bigg)^p, \\
        &\leq h^{p-1}_N \sum_{k = s_{1,N}}^{s_{2,N}-1} \sum_{j = 0}^{h_N-1} \int_{k/2^N}^{(k+1)/2^N} d^p_{\cX,p}(\sigma_N^{(k+j+1)/2^N}(\omega),\sigma_N^{(k+j)/2^N}(\omega))\d u, \\
        &\leq \frac{h^{p-1}_N}{2^N} \sum_{k = s_{1,N}}^{s_{2,N}-1} \sum_{j = 0}^{h_N-1}  d^p_{\cX,p}(\sigma_N^{(k+j+1)/2^N}(\omega),\sigma_N^{(k+j)/2^N}(\omega)),
        \end{align*}
    where we used Minkowski's inequality in the third-to-last line, and Hölder's inequality in the second-to-last one. It follows that
    \begin{align*}
        \int_{s_1}^{s_2} d^p_{\cX,p}(\sigma_N^{u+h}(\omega),\sigma_N^u(\omega))\d u   
        &\leq \frac{h^{p-1}_N}{2^N} \sum_{j = 0}^{h_N-1}\sum_{k = s_{1,N}}^{s_{2,N}-1}  d^p_{\cX,p}(\sigma_N^{(k+j+1)/2^N}(\omega),\sigma_N^{(k+j)/2^N}(\omega))\\
        &\leq \frac{h^{p-1}_N}{2^N} \sum_{j = 0}^{h_N-1}\sum_{k = s_{1,N}}^{s_{2,N}+h_N-1}  d^p_{\cX,p}(\sigma_N^{(k+1)/2^N}(\omega),\sigma_N^{k/2^N}(\omega)) \\
        &\leq \frac{h^{p-1}_N h_N}{2^N} \sum_{k = s_{1,N}}^{s_{2,N}+h_N-1}  d^p_{\cX,p}(\sigma_N^{(k+1)/2^N}(\omega),\sigma_N^{k/2^N}(\omega)) \\
        &= h^p\frac{h^{p}_N}{2^N}\frac{2^{Np}}{h^p_N}\sum_{k = s_{1,N}}^{s_{2,N}+h_N-1}  d^p_{\cX,p}(\sigma_N^{(k+1)/2^N}(\omega),\sigma_N^{k/2^N}(\omega)) \\
        &= h^p2^{N(p-1)}\sum_{k = s_{1,N}}^{s_{2,N}+h_N-1}  d^p_{\cX,p}(\sigma_N^{(k+1)/2^N}(\omega),\sigma_N^{k/2^N}(\omega)).
    \end{align*}
    This in turn yields
    \begin{align*}
        \int_{s_1}^{s_2} \frac{d^p_{\cX,p}(\sigma_N^{u+h}(\omega) ,\sigma_N^u(\omega))}{h^p}  \d u 
        &\leq 2^{N(p-1)} \sum_{k = s_{1,N}}^{s_{2,N}+h_N-1} d^p_{\cX,p}(X^{(k+1)/2^N}(\omega),X^{k/2^N}(\omega)).
    \end{align*}
    Integrating with respect to $\gamma_N$ then gives
    \begin{align*}
        \int_{L^p(\cX)} \int_{s_1}^{s_2} &\frac{d^p_{\cX,p}(f(u+h),f(u))}{h^p}\d u\, \eta_N(\d f) \\
        &=\int_{\Omega^{[N]}} \int_{s_1}^{s_2} \frac{d^p_{\cX,p}(\sigma_N^{u+h}(\omega),\sigma_N^u(\omega))}{h^p}\d u \,\gamma_N(\d \omega) \\
        &\leq 2^{N(p-1)}  \sum_{k = s_{1,N}}^{s_{2,N}+h_N-1} \int_{\Omega^{[N]}} d^p_{\cX,p}(X^{(k+1)/2^N}(\omega),X^{k/2^N}(\omega)) \gamma_N(\d \omega) \\
        &=  2^{N(p-1)} \sum_{k = s_{1,N}}^{s_{2,N}+h_N-1} \cA\cW^p_p (\X^{(k+1)/2^N},\X^{k/2^N}) \\
        &\leq  2^{N(p-1)}  \sum_{k = s_{1,N}}^{s_{2,N}+h_N-1} \bigg( \int_{k/2^N}^{(k+1)/2^N}|\X^\prime|_p(u)\d u \bigg)^p \\
        &\leq \sum_{k = s_{1,N}}^{s_{2,N}+h_N-1} \int_{k/2^N}^{(k+1)/2^N}|\X^\prime|_p^p(u)\d u \\
        &= \int_{s_1}^{s_2 + h}|\X^\prime_p|^p(u)\d u .
    \end{align*}
    We now let $N$ tend to infinity, and find by lower semi-continuity of the integrand that
    \begin{align*}
        \int_{L^p(\cX)} \int_{s_1}^{s_2} \frac{d^p_{\cX,p}(f(u+h),f(u))}{h^p}\d u\, \eta(\d f) 
        \leq \int_{s_1}^{s_2+h} |\X^\prime|^p_p(u)\d u.
    \end{align*}
    Since 
    \begin{align*}
        \int_{L^p(\cX)} \int_{s_1}^{s_2} & \frac{d^p_{\cX,p}(f(u+h),f(u))}{h^p}\d u\, \eta(\d f) \\
        &= \int_{W^{1,p}(\cX)} \int_{s_1}^{s_2} \frac{d^p_{\cX,p}(f(u+h),f(u))}{h^p}\d u\, \eta(\d f) \\
        &= \int_{C(\cX)} \int_{s_1}^{s_2} \frac{d^p_{\cX,p}((Tf)(u+h),(Tf)(u))}{h^p}\d u\, \eta(\d f) \\
        &= \int_{C(\cX)} \int_{s_1}^{s_2} \frac{d^p_{\cX,p}(f(u+h),f(u))}{h^p}\d u\, \tilde\eta(\d f),
    \end{align*}
    and since $\tilde\eta$ is concentrated on $AC(\cX)$, we can let $h$ tend to zero along dyadics, and find by Fatou's lemma that
    \begin{align} \label{eqn:bound_energy}
        \int_{C(\cX)} \int_{s_1}^{s_2} |f^\prime|^p(u)\d u\, \tilde\eta(\d f) 
        &\leq \int_{s_1}^{s_2} |\X^\prime|^p_p(u)\d u.
    \end{align}

    \medskip 
    {\bf Step 5: $\Ac\Wc_p(\X^u,\Y^u)=0.$} In order to show $\Ac\Wc_p(\X^u,\Y^u)=0,$ we need to verify that $\E f(\X^u)=\E f(\Y^u)$ for every countinuous and bounded function $f$ on the canonical space of filtered processes; see \Cref{lem:inform_process}. We observe that the function $g: u \longmapsto \E [f(\X^u)]$ is uniformly continuous on $[0,1]$, and thus the piecewise constant approximations  
    \[
        g_N(u)\coloneqq \sum_{i=0}^{2^N-1} g(i/2^N) \1_{[i/2^N,i+1/2^N)}(u),\quad u \in [0,1],
    \]
    converge uniformly to $g$ as $N\rightarrow \infty.$ Consequently, for every $\alpha \in C_b(\R)$, we have
    \[
        \lim_{N \rightarrow \infty} \int_0^1 \alpha(u)g_N(u) \mathrm{d}u=\int_0^1 \alpha(u)g(u) \mathrm{d}u = \int_0^1 \alpha(u)\E [f(\X^u)] \mathrm{d}u.
    \]
    Moreover, \eqref{eqn:AW_is_0} implies 
    \[ 
        \int_0^1 \alpha(u)g_N(u) \mathrm{d}u=\int_0^1 \alpha(u) \E[f(\Y^{N,u})]\mathrm{d}u,
    \]
    and therefore
    \begin{equation*}
        \int_0^1 \alpha(u)\E [f(\X^u)] \mathrm{d}u = \lim_{N \rightarrow \infty} \int_0^1 \alpha(u)g_N(u)\d u = \lim_{N \rightarrow \infty} \int_0^1 \alpha(u) \E[f(\Y^{N,u})] \mathrm{d}u. 
    \end{equation*}
    It is straightforward to verify that $\W \mapsto \int_0^1 \alpha(u) \E[f(\W^{u})] \mathrm{d}u$ is a continuous and bounded function on the canonical space of $\Fc\Pc_p( L^p(\Xc))$ and thus, by adapted weak convergence,
    \begin{equation*}
        \int_0^1 \alpha(u)\E [f(\X^u)] \mathrm{d}u = \lim_{N \rightarrow \infty} \int_0^1 \alpha(u) \E[f(\Y^{N,u})] \mathrm{d}u = \int_0^1 \alpha(u) \E[f(\Y^u)] \mathrm{d}u,
    \end{equation*}
    for every $\alpha \in C_b(\R)$.
    We conclude that $\E [f(\X^u)]=\E [f(\Y^u)]$ for almost every $u \in [0,T].$ Moreover, $u \mapsto \E [f(\X^u)]$ is continuous due to \Cref{lem:inform_process}. Further, since $\tilde\eta$ is concentrated on $AC^p(\Xc),$ we have that $u \mapsto \Y^u$ is absolutely continuous on $\Fc\Pc_p$, by \Cref{prop:AC1}. Note that the assumptions therein are satisfied thanks to \eqref{eqn:AW_is_0} and \eqref{eqn:bound_energy}. Consequently, using once again \Cref{lem:inform_process}, the map $u \mapsto \E [f(\Y^u)]$ is continuous as well. We conclude that
    $\E [f(\X^u)]=\E[ f(\Y^u)]$ for every $u \in [0,1]$, and thus $\Ac\Wc_p(\X^u,\Y^u)=0.$ This completes the proof.
\end{proof}

\appendix
\section{Appendix} \label{sec:appendix}

\subsection{Canonical space for filtered processes }\label{sec::canonical_filtered_process}

We briefly discuss the canonical space of filtered processes from \cite[Section 3.1]{BaBePa21}, which avoids the need for proper classes and allows us to remain within the framework of set theory when discussing filtered processes. More precisely, there is a correspondence between filtered processes and probability measures on a fixed filtered space. Given the complexity of this space and the fact that it is not our main focus, we merely outline the key ideas and refer the reader to \cite[Section 3.1]{BaBePa21} for further details.

Let $\X=\big(\Omega^\X,\Fc^\X,\F^\X,\P^\X,X \big)$ be a filtered process. We define the information process of $\X$ recursively backward in time by 
\[ {\rm ip}_T(\X)\coloneqq X_T,\quad {\rm ip}_t(\X)\coloneqq \big(X_t,{\rm Law}_{\P^\X}({\rm ip}_{t+1}(\X)\,\vert \Fc^\X_t) \big),\; t \in \{1,\ldots,T-1\}. \]
It follows from \cite[Theorem 3.10]{BaBePa21} that $\X$ can be identified with the law of ${\rm ip}_1(\X)$, as well as with the canonical representative $\overline{\X}$, which is constructed using the base space for the information process, denoted by $\Zc = \prod_{t=1}^T \Zc_t$, in the following sense: Let $\X^1$ and $\X^2$ be two filtered processes, and denote by $\overline{\X}^1$ and $\overline{\X}^2$ their respective canonical representatives. We have that
\begin{equation} \label{eqn:isometry} 
\Ac\Wc_p(\X^1,\X^2)=\Ac\Wc_p(\overline{\X}^1,\overline{\X}^2)=\Wc_p\big({\rm Law}( {\rm ip}_1(\X^1)),{\rm Law} ({\rm ip}_1(\X^2)) \big),\end{equation} 
where $\Wc_p\big({\rm Law} ({\rm ip}_1(\X^1)),{\rm Law} ({\rm ip}_1(\X^2)) \big)$ is the $p$-Wasserstein distance on $\Pc_p(\Zc_1)$, and $\Zc$ is equipped with an appropriate $p$-distance.
We refer to \cite[Section 3]{BaBePa21} for formal definitions of these objects.

\begin{remark}\label{rem::fp_as_sets}
    The collection of canonical filtered processes, denoted by $\CFP$ (resp.\ $\CFP_p$ for those with finite $p$-th moments) forms a set, and the adapted Wasserstein distance on this set defines a metric.\footnote{Here, we can, without loss of generality, replace $d$ with $d \wedge 1$ to avoid infinite values of the distance and ensure that we obtain a well-defined pseudo-metric.}
    Thus, whenever a construction requires proper sets or metric spaces, we can interpret $\FP$ and $\FP_p$ as $\CFP$ and $\CFP_p$, respectively. Similar challenges appear with the Gromov–Hausdorff distance and can also be alleviated in that setting as well; see \cite[Remark 7.2.5]{burago2001course}.
\end{remark}

As a direct consequence of \eqref{eqn:isometry}, we have that equivalence of two filtered processes can be tested with continuous and bounded functions of their information processes.

\begin{lemma} \label{lem:inform_process} Let $\X,\Y\in \Fc\Pc$. Then $\X \sim_{{\rm FP}} \Y$ if and only if for every continuous and bounded function $g: \Zc \rightarrow \R$ it holds that $\E[g(\X)]=\E[g(\Y)],$ where
\[\E[g(\X)] \coloneqq\int_{\Omega^\X} g({\rm ip}(\X)(\omega^\X)) \P^\X(\mathrm{d}\omega^\X) =\int_{\Zc} g(z) ({\rm ip}(\X)_{\#}\P^\X)(\mathrm{d}z). \]
Moreover, the map $\Fc\Pc \ni \X \mapsto \E[g(\X)]$ is continuous in the adapted weak topology.
\end{lemma}
\begin{proof} The statement is a direct consequence of \cite[Theorem 4.11]{BaBePa21} and the fact that $\Zc$ is a Polish space.
\end{proof}

\subsection{Metric space-valued \texorpdfstring{$L^p$}. functions}\label{sec_lp}

For a metric space $(\cY,d_\cY)$ and $p \in [1,\infty)$, we denote by $\cL^p(\cY)$ the set of Borel-measurable maps $f:[0,1] \longrightarrow \cY$ satisfying
\[
    \int_0^1 d^p_{\cY}(f(u),y_0)\d u < \infty
\]
for some (and hence all) $y_0 \in \cY$. We then endow $\cL^p(\cY)$ with the pseudo-metric
\[
    d_{\cL^p}(f,g) \coloneqq \bigg(\int_0^1 d^p_{\cY}(f(u),g(u))\d u\bigg)^{1/p},
\]
and denote by $L^p(\cY)$ the metric space consisting of the collection of the corresponding equivalence classes. If $\cY$ is complete, then so is $L^p(\cY)$; this follows the proof of \cite[Theorem~4.5]{legall2022measure}, replacing absolute values by the distance $d_\cY$ and $L^p$-norms by the distance $d_{\cL^\smalltext{p}}$. Moreover, if $\cY$ is separable, then so is $L^p(\cY)$; see the proof of \cite[Proposition~3.4.5]{cohn2013measure}.

We rely on the following compactness criterion; see \cite[Section~2.3]{lisini2007characterization} for further detailed references.

\begin{lemma}\label{lem::compactness_criterion}
    Let $p \in [1,\infty)$, and let $\cY$ be a metric space. A family $\mathscr{F} \subseteq L^p(\cY)$ is precompact, if
    \begin{enumerate}[label=(\roman*)]
        \item $\mathscr{F}$ is bounded,
        \item $\displaystyle \lim_{h \downarrow\downarrow 0} \sup_{f \in \mathscr{F}} \int_0^{T-h} d_{\cY}(f(u+h),f(u))\d u = 0$, and
        \item there exists $\Phi : \cY \longrightarrow [0,\infty]$, with compact sublevels $\{y \in \cY \,|\, \Phi(y) \leq c\}$, $c \in [0,\infty)$, such that
        \[
            \sup_{f \in \mathscr{F}} \int_0^T \Phi(f(u))\d u < \infty.
        \]
    \end{enumerate}
\end{lemma}

\bibliography{references}

\begin{thebibliography}{42}
\providecommand{\natexlab}[1]{#1}
\providecommand{\url}[1]{\texttt{#1}}
\expandafter\ifx\csname urlstyle\endcsname\relax
  \providecommand{\doi}[1]{doi: #1}\else
  \providecommand{\doi}{doi: \begingroup \urlstyle{rm}\Url}\fi

\bibitem[Acciaio et~al.(2020)Acciaio, Backhoff-Veraguas, and
  Zalashko]{AcBaZa20}
B.~Acciaio, J.~Backhoff-Veraguas, and A.~Zalashko.
\newblock Causal optimal transport and its links to enlargement of filtrations
  and continuous-time stochastic optimization.
\newblock \emph{Stochastic Processes and their Applications}, 130\penalty0
  (5):\penalty0 2918--2953, 2020.

\bibitem[Acciaio et~al.(2021)Acciaio, Beiglb{\"o}ck, and
  Pammer]{acciaio2021weak}
B.~Acciaio, M.~Beiglb{\"o}ck, and G.~Pammer.
\newblock Weak transport for non‐convex costs and model‐independence in a
  fixed‐income market.
\newblock \emph{Mathematical Finance}, 31\penalty0 (4):\penalty0 1423--1453,
  2021.

\bibitem[Acciaio et~al.(2024{\natexlab{a}})Acciaio, Kratsios, and
  Pammer]{acciaio2024designing}
B.~Acciaio, A.~Kratsios, and G.~Pammer.
\newblock Designing universal causal deep learning models: The geometric
  (hyper) transformer.
\newblock \emph{Mathematical Finance}, 34\penalty0 (2):\penalty0 671--735,
  2024{\natexlab{a}}.

\bibitem[Acciaio et~al.(2024{\natexlab{b}})Acciaio, Kr{\v{s}}ek, and
  Pammer]{AcKrPa23a}
B.~Acciaio, D.~Kr{\v{s}}ek, and G.~Pammer.
\newblock Multicausal transport: barycenters and dynamic matching.
\newblock \emph{arXiv preprint arXiv:2401.12748}, 2024{\natexlab{b}}.

\bibitem[Ambrosio and Gigli(2013)]{Ambrosio2013users}
L.~Ambrosio and N.~Gigli.
\newblock A user’s guide to optimal transport. modelling and optimisation of
  flows on networks, 1--155.
\newblock \emph{Lecture Notes in Math}, 2062, 2013.

\bibitem[Ambrosio et~al.(2008)Ambrosio, Gigli, and
  Savar{\'e}]{ambrosio2008gradient}
L.~Ambrosio, N.~Gigli, and G.~Savar{\'e}.
\newblock \emph{Gradient flows: in metric spaces and in the space of
  probability measures}.
\newblock Springer Science \& Business Media, 2008.

\bibitem[Backhoff-Veraguas et~al.(2020{\natexlab{a}})Backhoff-Veraguas, Bartl,
  Beiglb{\"o}ck, and Eder]{BaBaBeEd19a}
J.~Backhoff-Veraguas, D.~Bartl, M.~Beiglb{\"o}ck, and M.~Eder.
\newblock Adapted {W}asserstein distances and stability in mathematical
  finance.
\newblock \emph{Finance and Stochastics}, 24\penalty0 (3):\penalty0 601--632,
  2020{\natexlab{a}}.

\bibitem[Backhoff-Veraguas et~al.(2020{\natexlab{b}})Backhoff-Veraguas, Bartl,
  Beiglb{\"o}ck, and Eder]{BaBaBeEd19b}
J.~Backhoff-Veraguas, D.~Bartl, M.~Beiglb{\"o}ck, and M.~Eder.
\newblock All adapted topologies are equal.
\newblock \emph{Probability Theory and Related Fields}, 178:\penalty0
  1125--1172, 2020{\natexlab{b}}.

\bibitem[Bartl and Wiesel(2023)]{BaWi23}
D.~Bartl and J.~Wiesel.
\newblock Sensitivity of multiperiod optimization problems with respect to the
  adapted {W}asserstein distance.
\newblock \emph{SIAM Journal on Financial Mathematics}, 14\penalty0
  (2):\penalty0 704--720, 2023.

\bibitem[Bartl et~al.(2024)Bartl, Beiglb{\"o}ck, and Pammer]{BaBePa21}
D.~Bartl, M.~Beiglb{\"o}ck, and G.~Pammer.
\newblock The {W}asserstein space of stochastic processes.
\newblock \emph{Journal of the European Mathematical Society}, 2024.

\bibitem[Bayraktar and Han(2023)]{BaHa23}
E.~Bayraktar and B.~Han.
\newblock Fitted value iteration methods for bicausal optimal transport.
\newblock \emph{arXiv preprint arXiv:2306.12658}, 2023.

\bibitem[Beiglb{\"o}ck et~al.(2024)Beiglb{\"o}ck, Pfl{\"u}gl, and
  Schrott]{beiglbock2024probabilistic}
M.~Beiglb{\"o}ck, S.~Pfl{\"u}gl, and S.~Schrott.
\newblock A probabilistic view on the adapted {W}asserstein distance.
\newblock \emph{arXiv preprint arXiv:2406.19810}, 2024.

\bibitem[Bion-Nadal and Talay(2019)]{BiTa19}
J.~Bion-Nadal and D.~Talay.
\newblock On a {W}asserstein-type distance between solutions to stochastic
  differential equations.
\newblock \emph{Ann. Appl. Probab.}, 29\penalty0 (3):\penalty0 1609--1639,
  2019.

\bibitem[Bogachev(2007)]{bogachev2007measure}
V.~I. Bogachev.
\newblock \emph{Measure theory. {V}ol. {I}, {II}}.
\newblock Springer, 2007.

\bibitem[Bonnier et~al.(2023)Bonnier, Liu, and Oberhauser]{BoLiOb23}
P.~Bonnier, C.~Liu, and H.~Oberhauser.
\newblock Adapted topologies and higher rank signatures.
\newblock \emph{The Annals of Applied Probability}, 33\penalty0 (3):\penalty0
  2136--2175, 2023.

\bibitem[Burago et~al.(2001)Burago, Burago, and Ivanov]{burago2001course}
D.~Burago, Y.~Burago, and S.~Ivanov.
\newblock \emph{A course in metric geometry}.
\newblock American Mathematical Society, 2001.

\bibitem[Cohn(2013)]{cohn2013measure}
D.~L. Cohn.
\newblock \emph{Measure theory}.
\newblock Birkh\"{a}user/Springer, second edition, 2013.

\bibitem[Eckstein and Pammer(2024)]{eckstein2023computational}
S.~Eckstein and G.~Pammer.
\newblock Computational methods for adapted optimal transport.
\newblock \emph{Ann. Appl. Probab.}, 34\penalty0 (1A):\penalty0 675--713, 2024.

\bibitem[Hoover(1991)]{hoover1991convergence}
D.~N. Hoover.
\newblock Convergence in distribution and {S}korokhod convergence for the
  general theory of processes.
\newblock \emph{Probability theory and related fields}, 89:\penalty0 239--259,
  1991.

\bibitem[Horvath et~al.(2023)Horvath, Lemercier, Liu, Lyons, and
  Salvi]{HoLeLiLySa23}
B.~Horvath, M.~Lemercier, C.~Liu, T.~Lyons, and C.~Salvi.
\newblock Optimal stopping via distribution regression: a higher rank signature
  approach.
\newblock \emph{arXiv preprint arXiv:2304.01479}, 2023.

\bibitem[Jiang and Lim(2025)]{jiang2025transfer}
Y.~Jiang and F.~R. Lim.
\newblock A transfer principle for computing the adapted wasserstein distance
  between stochastic processes.
\newblock \emph{arXiv preprint arXiv:2505.21337}, 2025.

\bibitem[Jiang and Obloj(2024)]{jiang2024sensitivity}
Y.~Jiang and J.~Obloj.
\newblock Sensitivity of causal distributionally robust optimization.
\newblock \emph{arXiv preprint arXiv:2408.17109}, 2024.

\bibitem[Jordan et~al.(1998)Jordan, Kinderlehrer, and
  Otto]{jordan1998variational}
R.~Jordan, D.~Kinderlehrer, and F.~Otto.
\newblock The variational formulation of the {F}okker--{P}lanck equation.
\newblock \emph{SIAM journal on mathematical analysis}, 29\penalty0
  (1):\penalty0 1--17, 1998.

\bibitem[Jourdain and Pammer(2024)]{JoPa23}
B.~Jourdain and G.~Pammer.
\newblock An extension of martingale transport and stability in robust finance.
\newblock \emph{Electronic Journal of Probability}, 29:\penalty0 1--30, 2024.

\bibitem[Lacker(2023)]{lacker2023independent}
D.~Lacker.
\newblock Independent projections of diffusions: Gradient flows for variational
  inference and optimal mean field approximations.
\newblock \emph{arXiv preprint arXiv:2309.13332}, 2023.

\bibitem[Lambert et~al.(2022)Lambert, Chewi, Bach, Bonnabel, and
  Rigollet]{lambert2022variational}
M.~Lambert, S.~Chewi, F.~Bach, S.~Bonnabel, and P.~Rigollet.
\newblock Variational inference via {W}asserstein gradient flows.
\newblock \emph{Advances in Neural Information Processing Systems},
  35:\penalty0 14434--14447, 2022.

\bibitem[Le~Gall(2022)]{legall2022measure}
J.-F. Le~Gall.
\newblock \emph{Measure theory, probability, and stochastic processes}.
\newblock Springer, 2022.

\bibitem[Lisini(2007)]{lisini2007characterization}
S.~Lisini.
\newblock Characterization of absolutely continuous curves in {W}asserstein
  spaces.
\newblock \emph{Calc. Var. Partial Differential Equations}, 28\penalty0
  (1):\penalty0 85--120, 2007.

\bibitem[Lisini(2016)]{lisini2016absolutely}
S.~Lisini.
\newblock Absolutely continuous curves in extended {W}asserstein--{O}rlicz
  spaces.
\newblock \emph{ESAIM: Control, Optimisation and Calculus of Variations},
  22\penalty0 (3):\penalty0 670--687, 2016.

\bibitem[Mielke et~al.(2014)Mielke, Peletier, and Renger]{mielke2004}
A.~Mielke, M.~A. Peletier, and D.~R.~M. Renger.
\newblock On the relation between gradient flows and the large-deviation
  principle, with applications to {M}arkov chains and diffusion.
\newblock \emph{Potential Analysis}, 41:\penalty0 1293--1327, 2014.

\bibitem[Mokrov et~al.(2021)Mokrov, Korotin, Li, Genevay, Solomon, and
  Burnaev]{mokrov2021large}
P.~Mokrov, A.~Korotin, L.~Li, A.~Genevay, J.~M. Solomon, and E.~Burnaev.
\newblock Large-scale {W}asserstein gradient flows.
\newblock \emph{Advances in Neural Information Processing Systems},
  34:\penalty0 15243--15256, 2021.

\bibitem[Nielsen and Sun(2021)]{NiSu20}
F.~Nielsen and K.~Sun.
\newblock Chain rule optimal transport.
\newblock In \emph{Progress in Information Geometry: Theory and Applications},
  pages 191--217. Springer, 2021.

\bibitem[Otto(2001)]{otto2001}
F.~Otto.
\newblock The geometry of dissipative evolution equations: The porous medium
  equation.
\newblock \emph{Communications in Partial Differential Equations}, 26\penalty0
  (1-2):\penalty0 101--174, 2001.

\bibitem[Pflug and Pichler(2014)]{pflug2014multistage}
G.~C. Pflug and A.~Pichler.
\newblock \emph{{M}ultistage {S}tochastic {O}ptimization}.
\newblock Springer, 2014.

\bibitem[Pflug and Pichler(2012)]{PfPi12}
G.~Ch. Pflug and A.~Pichler.
\newblock A distance for multistage stochastic optimization models.
\newblock \emph{SIAM Journal on Optimization}, 22\penalty0 (1):\penalty0 1--23,
  2012.

\bibitem[Pinzi and Savar\'{e}(2025)]{pinzi2025}
A.~Pinzi and G.~Savar\'{e}.
\newblock Nested superposition principle: From the continuity equation on
  random measures to interacting particle systems.
\newblock \emph{In preparation}, 2025.

\bibitem[Sauldubois and Touzi(2024)]{sauldubois2024first}
N.~Sauldubois and N.~Touzi.
\newblock First order martingale model risk and semi-static hedging.
\newblock \emph{arXiv preprint arXiv:2410.06906}, 2024.

\bibitem[Stepanov and Trevisan(2017)]{stepanov2017three}
E.~Stepanov and D.~Trevisan.
\newblock Three superposition principles: currents, continuity equations and
  curves of measures.
\newblock \emph{Journal of Functional Analysis}, 272\penalty0 (3):\penalty0
  1044--1103, 2017.

\bibitem[Vershik(1970)]{Ve70}
A.~M. Vershik.
\newblock Decreasing sequences of measurable partitions and their applications.
\newblock \emph{Sov. Mat. Dokl.}, 11\penalty0 (4):\penalty0 1007 -- 1011, 1970.

\bibitem[Vershik(1994)]{Ve94}
A.~M. Vershik.
\newblock Theory of decreasing sequences of measurable partitions.
\newblock \emph{Algebra i Analiz}, 6\penalty0 (4):\penalty0 1--68, 1994.

\bibitem[Villani(2009)]{villani2009optimal}
C.~Villani.
\newblock \emph{{O}ptimal {T}ransport}.
\newblock Springer, 2009.

\bibitem[Yan et~al.(2024)Yan, Wang, and Rigollet]{yan2024learning}
Y.~Yan, K.~Wang, and P.~Rigollet.
\newblock Learning {G}aussian mixtures using the {W}asserstein--{F}isher--{R}ao
  gradient flow.
\newblock \emph{The Annals of Statistics}, 52\penalty0 (4):\penalty0
  1774--1795, 2024.

\end{thebibliography}

\end{document}